\documentclass[a4paper,11pt,leqno]{article}

\usepackage{amsmath,amsthm,amssymb,mathtools,cases,mathrsfs,amscd}
\usepackage[all]{xy}

\newtheorem{theo}{Theorem}[section]
\newtheorem{prop}[theo]{Proposition}
\newtheorem{lemm}[theo]{Lemma}
\newtheorem{coro}[theo]{Corollary}
\newtheorem{conj}[theo]{Conjecture}

\theoremstyle{definition}
\newtheorem{defi}[theo]{Definition}
\newtheorem{exam}[theo]{Example}
\newtheorem*{ackn}{Acknowledgments}
\newtheorem{rema}[theo]{Remark}

\makeatletter

\@addtoreset{equation}{section}
\makeatother

\newcommand{\Hom}{\mathop{\mathrm{Hom}}\nolimits}
\newcommand{\End}{\mathop{\mathrm{End}}\nolimits}
\newcommand{\Spec}{\mathop{\mathrm{Spec}}\nolimits}
\newcommand{\supp}{\mathop{\mathrm{supp}}\nolimits}
\newcommand{\Frac}{\mathop{\mathrm{Frac}}\nolimits}
\newcommand{\tdeg}{\mathop{\mathrm{tr.deg}}\nolimits}
\newcommand{\Aut}{\mathop{\mathrm{Aut}}\nolimits}
\newcommand{\aus}{\mathop{\mathrm{Aut}_{\sigma}}\nolimits}
\newcommand{\Mat}{\mathop{\mathrm{Mat}}\nolimits}
\newcommand{\GL}{\mathop{\mathrm{GL}}\nolimits}
\newcommand{\im}{\mathop{\mathrm{im}}\nolimits}

\newcommand{\Rep}{\mathop{\mathbf{Rep}}\nolimits}
\newcommand{\vsp}{\mathop{\mathbf{Vec}}\nolimits}
\newcommand{\Mod}{\mathop{\mathbf{Mod}}\nolimits}

\newcommand{\id}{\mathrm{id}}
\newcommand{\Gal}{\mathop{\mathrm{Gal}}\nolimits}
\newcommand{\ev}{\mathrm{ev}}
\newcommand{\Lie}{\mathop{\mathrm{Lie}}\nolimits}

\newcommand{\mfa}{\mathfrak{a}}
\newcommand{\mfb}{\mathfrak{b}}
\newcommand{\mfm}{\mathfrak{m}}
\newcommand{\mfp}{\mathfrak{p}}
\newcommand{\mfq}{\mathfrak{q}}

\newcommand{\mbf}{\mathbf{f}}
\newcommand{\mbh}{\mathbf{h}}
\newcommand{\mbm}{\mathbf{m}}
\newcommand{\mbn}{\mathbf{n}}
\newcommand{\mbu}{\mathbf{u}}

\newcommand{\mbx}{\mathbf{x}}

\newcommand{\Fq}{\mathbb{F}_{q}}
\newcommand{\ftv}{\mathbb{F}_{q}(t)_{v}}
\newcommand{\ktv}{K(t)_{v}}
\newcommand{\kstv}{K^{\mathrm{sep}}(t)_{v}}
\newcommand{\kstl}{K^{\mathrm{sep}}(\!(t-\lambda_{l})\!)}
\newcommand{\ks}{K^{\mathrm{sep}}}

\newcommand{\kv}{K_{v(\theta)}}
\newcommand{\kll}{K_{\lambda_{l}}}
\newcommand{\cll}{\mathbb{C}_{\lambda_{l}}}

\newcommand{\phm}{\Phi {\rm M}}
\newcommand{\epm}{\Phi {\rm M}^{{\rm \acute{e}t}}}

\newcommand{\TM}{\mathcal{T}_{M}}

\newcommand{\FX}{F[X, \Delta^{-1}]}
\newcommand{\EX}{E[X, \Delta^{-1}]}
\newcommand{\SX}{\Sigma[X, \Delta^{-1}]}
\newcommand{\LX}{L[X, \Delta^{-1}]}
\newcommand{\SlX}{\Sigma_{l}[X, \Delta^{-1}]}
\newcommand{\LlX}{L_{l}[X, \Delta^{-1}]}
\newcommand{\X}{[X, \Delta^{-1}]}

\newcommand{\vp}{\varphi}
\newcommand{\Sl}{\Sigma_{l}}
\newcommand{\Slz}{\Sigma_{l_{0}}}
\newcommand{\ali}{\alpha_{l}^{-1}}
\newcommand{\rar}{\rightarrow}
\newcommand{\La}{\Lambda}
\newcommand{\Lal}{\Lambda_{l}}
\newcommand{\Af}{\mathbb{A}^{1}}
\newcommand{\RR}{^{(R)}}

\title{On $v$-adic periods of $t$-motives}
\author{Yoshinori Mishiba\footnote{Graduate School of Mathematics, Kyushu University,
744, Motooka, Nishi-ku, Fukuoka, 819-0395, JAPAN
\endgraf e-mail: y-mishiba@math.kyushu-u.ac.jp}}

\begin{document}
\maketitle

\begin{abstract}
In this paper,
we prove the equality between
the transcendental degree of the field generated by the $v$-adic periods of a $t$-motive $M$
and the dimension of the Tannakian Galois group for $M$,
where $v$ is a \lq\lq finite'' place of the rational function field over a finite field.
As an application, we prove the algebraic independence of certain \lq\lq formal'' polylogarithms. 
\end{abstract}

\tableofcontents

\section{Introduction}
Let $\Fq$ be the finite field with $q$ elements,
$\theta$ and $t$ be variables independent from each other,
and $v \in \Fq[t]$ a fixed monic irreducible polynomial of degree $d$.
Let $M$ be a rigid analytically trivial $t$-motive over $\overline{\Fq(\theta)}$.
Then there exists an $\infty$-adic period matrix for the Betti realization of $M$.
Set $\Lambda$ to be the field generated by the components of this matrix over $\overline{\Fq(\theta)}(t)$.
Set $\Gamma$ to be the Tannakian Galois group of $M$ with respect to the Betti realization.
Papanikolas \cite{Papa} shows that the transcendental degree of $\Lambda$ over $\overline{\Fq(\theta)}(t)$
coincides with the dimension of $\Gamma$.
In this paper, we prove the $v$-adic analogue of this theorem.

Let $K / \Fq$ be a regular extension of fields.
We set $\ks[t]_{v} := \varprojlim(\ks[t]/v^n)$ and $\kstv := \Fq(t) \otimes_{\Fq[t]} \ks[t]_{v}$,
where $\ks$ is a separable closure of $K$.
We also define $\ftv$ and $\ktv$ by the same way.
Let $\sigma$ be the ring endomorphism $\sum a_i t^i \mapsto \sum a_i^q t^i$ of $\ks[t]$.
Then $\sigma$ naturally extends to an endomorphism of $\kstv$, also denoted by $\sigma$.
A $\vp$-module over $\ktv$ is a pair $(M, \vp)$ (or simply $M$)
where $M$ is a $\ktv$-vector space and $\vp : M \rar M$ is an additive map
such that $\vp(ax) = \sigma(a) \vp(x)$ for all $a \in \ktv$ and $x \in M$.
A morphism of $\vp$-modules is a $\ktv$-linear map which is compatible with the $\vp$'s.
A tensor product of two $\vp$-modules is defined naturally.

For any $\vp$-module $M$, we define the $v$-adic realization of $M$:
$$V(M) := (\kstv \otimes_{\ktv} M)^{\vp},$$
where $\vp$ acts on $\kstv \otimes_{\ktv} M$ by $\sigma \otimes \vp$ and $(-)^{\vp}$ is the $\vp$-fixed part.
Then there exists a natural map
$$\iota_{M} : \kstv \otimes_{\ftv} V(M) \rar \kstv \otimes_{\ktv} M.$$
We can prove that $\iota_{M}$ is injective for each $\vp$-module $M$.
A $\vp$-module $M$ is said to be $\kstv$-trivial
if $M$ is finite-dimensional over $\ktv$ and $\iota_{M}$ is an isomorphism.
Then the category of $\kstv$-trivial $\vp$-modules over $\ktv$
equipped with the functor $V$ forms a neutral Tannakian category over $\ftv$.
For any $\kstv$-trivial $\varphi$-module $M$, we denote by $\Gamma_{M}$
the Tannakian Galois group of the Tannakian subcategory of $\kstv$-trivial $\vp$-modules generated by $M$
(see Subsections \ref{phi_l} and \ref{phi_v}).

Let $M$ be a finite-dimensional $\vp$-module and $\mbm \in \Mat_{r \times 1}(M)$ a $\ktv$-basis of $M$.
Then there exists a matrix $\Phi \in \Mat_{r \times r}(\ktv)$ such that $\vp \mbm = \Phi \mbm$.
If $M$ is $\kstv$-trivial, we can take a matrix $\Psi = (\Psi_{ij})_{i,j} \in \GL_{r}(\kstv)$
such that $\Psi^{-1} \mbm$ forms an $\ftv$-basis of $V(M)$.
The entries of this matrix are called $v$-adic periods of $M$,
which are our main objects of study in this paper.
We set
$$\Sigma := \ktv[\Psi,1/\det\Psi] := \ktv[\Psi_{11},\Psi_{12},\dots,\Psi_{rr},1/\det\Psi] \subset \kstv.$$
Then $\Sigma$ is stable under the $\sigma$-action.
For any $\ftv$-algebra $R$ and $S$, we set $S\RR := R \otimes_{\ftv} S$.
If $\sigma$ acts on $S$, we define the $\sigma$-action on $S\RR$ by $\id \otimes \sigma$.
Set $\Gamma(R) := \aus(\Sigma\RR / \ktv\RR)$ the group of automorphisms of $\Sigma\RR$ over $\ktv\RR$ that commute with $\sigma$.
Then $\Gamma$ forms a functor from the category of $\ftv$-algebras to the category of groups.
If we factorize $v = \prod_{l \in \mathbb{Z}/d}(t-\lambda_{l})$
in $K^{\mathrm{sep}}[t]$ with $\lambda_{l}^{q} = \lambda_{l+1}$,
then we can write $\kstv = \prod_{l} \kstl$ and $\Psi_{ij} = (\Psi_{ijl})_{l}$ where $\Psi_{ijl} \in \kstl$.
We set
$$\Lal := \ktv(\Psi_{11l},\dots,\Psi_{rrl}) \subset \kstl$$
for each $l \in \mathbb{Z}/d$.
Our main result in this paper is (see Lemma \ref{gal_g=a} and Theorems \ref{gal_smooth} and \ref{p-d_v-isom}):
\begin{theo}\label{int_main}
The functor $\Gamma$ is representable by a smooth affine algebraic variety over $\ftv$, also denoted by $\Gamma$.
We have an equality $\dim \Gamma = \tdeg_{\ktv} \Lal$ for each $l \in \mathbb{Z}/d$
and there exists a natural isomorphism $\Gamma \rar \Gamma_{M}$ of affine group schemes over $\ftv$.
\end{theo}

This theorem is a $v$-adic analogue of Papanikolas's Theorem 4.3.1 and 4.5.10 in \cite{Papa}, which treated $\infty$-adic objects.
The proof of this theorem follows \cite{Papa} closely,
but since $\kstv$ is not a field if $d > 1$, several arguments here are more complicated than in \cite{Papa}.
Let $K = \Fq(\theta)$ where $\theta$ is a variable independent of $t$.
Papanikolas shows the equality of the transcendental degree of the field of periods (specialized at $t = \theta$) over $K$
and the dimension of the Tannakian Galois group using the so-called ABP-criterion proved by Anderson, Brownawell and Papanikolas in \cite{ABP}.
In fact he proved an algebraic independence of Carlitz logarithms.
On the other hand, Anderson and Thakur \cite{AnTh} shows that the relation between the Carlitz zeta values and Carlitz logarithms.
Then using these results, Chang and Yu \cite{ChYu} determined the all algebraic relations among the Carlitz zeta values.
These applications are our motivation of this paper,
but in this paper, we can only prove a $v$-adic analogue of the ABP-criterion for the rank one case.

In Section \ref{sec_phi-t}, first we review a theory of $\vp$-modules in a general setting and construct a Tannakian category.
In the $v$-adic case, we show that this category is equivalent to the category of Galois representations.
In Section \ref{sec_gal}, we consider Frobenius equations in our situation, and construct $\Gamma$.
In Section \ref{sec_p-d}, we discuss the relation between $\Gamma$ and $\Gamma_{M}$,
and prove that these are isomorphic in the $v$-adic case.
This uses the fact that the set of $\ftv$-valued points $\Gamma(\ftv)$ is Zariski dense in $\Gamma$.
Since $\Gamma(\ftv)$ contains the Galois image, this is large enough in $\Gamma$.
This is an essentially different point from Papanikolas's proof for the $\infty$-adic case,
in which the Zariski density is not proved and other facts are used to show this isomorphism.
In Section \ref{sec_abp}, we discuss a $v$-adic analogue of the ABP-criterion.
In Section \ref{sec_log}, we prove the algebraic independence of certain \lq\lq formal'' polylogarithms.

\begin{ackn}
The author thanks Yuichiro Taguchi for many helpful discussions on the contents of this paper
and for reading preliminary manuscripts of this paper carefully.
\end{ackn}

\section{Notations and terminology}
\subsection{Table of symbols}
\begin{tabular}{lcp{29em}}
$\Fq$ & := & the finite field of $q$ elements \\
$\bar{k}$ & := & an algebraic closure of a field $k$ \\
$k^{\mathrm{sep}}$ & := & the separable closure of a field $k$ in $\bar{k}$ \\
$\# S$ & := & the cardinality of a set $S$ \\
$\Mat_{r \times s}(R)$ & := & the set of $r$ by $s$ matrices with entries in a ring or module $R$ \\
$\GL_{r}(R)$ & := & the group of invertible $r$ by $r$ matrices with entries in a ring $R$ \\
$\vsp(k)$ & := & the category of finite-dimensional vector spaces over a field $k$ \\
$\Rep(G, R)$ & := & for a ring $R$ the category of finitely generated $R$-representations of an affine group scheme $G$ over $R$,
or for a topological ring $R$ the category of finitely generated continuous $R$-representations of a topological group $G$
\end{tabular}

\subsection{Action}
Let $R$ be a ring or module and $f : R \rar R$ a map.
For a matrix $A = (A_{ij})_{ij}\in \Mat_{r \times s}(R)$, we denote by $f(A)$ the matrix $(f(A_{ij}))_{ij}$.

Let $S$ be a set and $H$ a set of maps from $S$ to itself.
Then we denote by $S^{H}$ the subset of $S$ consisting of elements which are fixed by $H$.
For a map $f : S \rar S$, we set $S^{f} := S^{\{f\}}$.

\subsection{Base change}
Let $R \rar S$ be a homomorphism of commutative rings and $X$ a scheme over $R$.
We denote by $X_{S} := X \times_{\Spec R} \Spec S$ the base change from $R$ to $S$ of $X$.
We also denote by $X(S) := \Hom_{\Spec R}(\Spec S, X)$ the set of $S$-valued points of $X$ over $R$.
When $R$ and $S$ are fields, we have a natural injection $X(R) \hookrightarrow X(S) \cong X_{S}(S)$.
We always consider $X(R)$ as a subset of $X_{S}(S)$ via this injection.

\section{$\vp$-modules}\label{sec_phi-t}

\subsection{\'etale $\vp$-modules}
In this subsection, we recall the definitions and properties of \'etale $\vp$-modules (cf.\ \cite{Font}).
Let $A$ be a commutative ring and $\sigma$ an endomorphism of $A$.
For any $A$-module $M$, we put $M^{(\sigma)} := A \otimes_{A} M$, the scalar extension of $M$ by $\sigma$.
A map $\vp:M \rar M$ is said to be $\sigma$\textit{-semilinear} if
$\vp(x+y) = \vp(x)+\vp(y)$ and $\vp(ax) = \sigma(a)\vp(x)$ for all $x, y \in M$ and $a \in A$.
Then it is clear that to give a $\sigma$-semilinear map $\vp : M \rar M$ is equivalent
to giving an $A$-linear map $\vp_{\sigma} : M^{(\sigma)} \rar M$.

\begin{defi}
A \textit{$\vp$-module} $(M, \vp)$ over $(A, \sigma)$ (or simply, $M$ over $A$) is an $A$-module $M$
endowed with a $\sigma$-semilinear map $\vp : M \rar M$.
A \textit{morphism} of $\vp$-modules is an $A$-linear map which is compatible with the $\vp$'s.
When $A$ is a noetherian ring, a $\vp$-module $(M, \vp)$ is said to be \textit{\'etale}
if $M$ is a finitely generated $A$-module and $\vp_{\sigma} : M^{(\sigma)} \rar M$ is bijective.
\end{defi}

We denote by $\phm_{A}$ the category of $\vp$-modules over $A$
and $\epm_{A}$ its full subcategory consisting of all \'etale $\vp$-modules.
For any $\vp$-modules $M$ and $N$, we denote by $\Hom_{\vp}(M, N)$ the set of morphisms of $M$ to $N$ in $\phm_{A}$.

Let $A_{\sigma}[\vp]$ be the ring (non commutative if $\sigma \neq \id_{A}$)
generated by $A$ and an element $\vp$ with the relation
$$\vp a = \sigma(a) \vp$$
for each $a \in A$.
Then it is clear that the category $\phm_{A}$ and the category of $A_{\sigma}[\vp]$-modules
are naturally identified.
Hence, the category $\phm$ is an $A^{\sigma}$-linear abelian category.

For each $\vp$-module $M$ and $N$, we denote by $M \otimes N$ the \textit{tensor product} of $M$ and $N$,
which is $M \otimes_{A} N$ as an $A$-module and has a $\vp$-action defined by $\vp \otimes \vp$.
Then the functor $\otimes$ is a bi-additive functor
and $(A, \sigma)$ is an identity object in $\phm_{A}$ for this tensor product.
Therefore the category $\phm_{A}$ is an abelian tensor category (\cite{DeMi}).

\begin{prop}\label{phi_endo}
There exists a natural isomorphism $A^{\sigma} \cong \End_{\vp}(A) := \Hom_{\vp}(A, A)$.
\end{prop}

\begin{proof}
For any endomorphism $f \in \End_{\vp}(A)$,
we have $\sigma(f(1)) = \vp(f(1)) = f(\vp(1)) = f(\sigma(1)) = f(1)$.
Hence $f(1) \in A^{\sigma}$.
Conversely for any element $a \in A^{\sigma}$,
we have a map $f_{a} : A \rar A ; x \mapsto ax$.
It is clear that $f_{a} \in \End_{\vp}(A)$.
These are inverse to each other.
\end{proof}

\begin{prop}\label{phi_etale}
Assume that $A$ is noetherian and $\sigma$ is flat.
Then the category $\epm_{A}$ is an abelian $A^{\sigma}$-linear tensor category.
\end{prop}

\begin{proof}
It is clear that $\epm_{A}$ is closed under finite sums and tensor products,
and the identity object $(A, \sigma)$ is \'etale.
Therefore it is enough to show that for each \'etale $\vp$-modules $M$ and $N$ and a morphism $f : M \rar N$,
the kernel and cokernel of $f$ in $\phm_{A}$ are \'etale.
Since $M$ and $N$ are \'etale and $\sigma$ is flat, we have the commutative diagram
\[\xymatrix{
0 \ar[r] & (\ker f)^{(\sigma)} \ar[r] \ar[d]_{\vp'_{\sigma}} & M^{(\sigma)} \ar[r] \ar[d]_{\vp_{M, \sigma}}
& N^{(\sigma)} \ar[r] \ar[d]_{\vp_{N, \sigma}} & (\im f)^{(\sigma)} \ar[r] \ar[d]_{\vp''_{\sigma}} & 0 \\
0 \ar[r] & \ker f \ar[r] & M \ar[r] & N \ar[r] &  \im f \ar[r] & 0, \\
}\]
where $\vp_{M, \sigma}$ and $\vp_{N, \sigma}$ are isomorphisms and the rows are exact.
Then we have that $\vp'_{\sigma}$ and $\vp''_{\sigma}$ are isomorphism by a diagram chasing.
\end{proof}

Let $(M, \vp_{M})$ and $(N, \vp_{N})$ be $\vp$-modules over $A$.
If $\vp_{M,\sigma} : M^{(\sigma)} \rar M$ is an isomorphism, we define a $\vp$-module $\Hom(M, N)$,
whose underlying $A$-module is the space $\Hom_{A}(M, N)$ of $A$-module homomorphisms
and a $\vp$-action is defined by
$$\Hom_{A}(M, N)^{(\sigma)} \rar \Hom_{A}(M^{(\sigma)}, N^{(\sigma)}) \rar \Hom_{A}(M, N),$$
where the first map is the natural map and the second map is defined by
$f \mapsto \vp_{N, \sigma} \circ f \circ \vp_{M, \sigma}^{-1}$.
There exists a natural morphism of $\vp$-modules $\ev_{M, N} : \Hom(M, N) \otimes M \rar N$.
For each $M$ such that $\vp_{M,\sigma}$ is an isomorphism, we set $M^{\vee} := \Hom(M, A)$ the dual of $M$.

\begin{prop}\label{phi_i-hom}
Assume that $A$ is noetherian and $\sigma$ is flat.
Then for any objects $M$ and $N$ in $\epm_{A}$, the $\vp$-module $\Hom(M, N)$ is \'etale,
the contravariant functor
$$\epm_{A} \rar \mathbf{Set} ; \ T \mapsto \Hom_{\vp}(T \otimes M, N)$$
is representable by $\Hom(M, N)$
and $\ev_{M, N}$ corresponds to $\id_{\Hom(M, N)}$.
\end{prop}

\begin{proof}
Since $M$ and $N$ are finitely generated and $A$ is noetherian, $\Hom(M, N)$ is also finitely generated.
Since $\sigma$ is flat and $M$ is finitely presented,
the map $\Hom_{A}(M, N)^{(\sigma)} \rar \Hom_{A}(M^{(\sigma)}, N^{(\sigma)})$ is an isomorphism (\cite{Bour}, Chap.\ I, Sect.\ 2, Prop.\ 11).
Since $\vp_{M, \sigma}$ and $\vp_{N, \sigma}$ are bijective, the $\vp$-module $\Hom(M, N)$ is \'etale.
It is clear that there exists a natural isomorphism
$\Hom_{A}(T \otimes M, N) \cong \Hom_{A}(T, \Hom(M, N))$
which is functorial in $T$.
Then we can calculate that the subspaces $\Hom_{\vp}(T \otimes M, N)$ and $\Hom_{\vp}(T, \Hom(M, N))$ are corresponding with this isomorphism.
The last assertion is clear.
\end{proof}

\begin{prop}\label{phi_rigid}
Assume that $A$ is a field.
Then the category $\epm_{A}$ is a rigid abelian $A^{\sigma}$-linear tensor category.
\end{prop}

\begin{proof}
By Proposition \ref{phi_etale}, $\epm_{A}$ is an abelian $A^{\sigma}$-linear tensor category.
By Proposition \ref{phi_i-hom}, $\epm_{A}$ has internal homs.
Therefore it is enough to show that the natural map
$$\otimes_{i \in I} \Hom(M_{i}, N_{i}) \rar \Hom(\otimes_{i \in I} M_{i}, \otimes_{i \in I} N_{i})$$
is an isomorphism for any finite families of objects $(M_{i})_{i \in I}$ and $(Y_{i})_{i \in I}$,
and the natural map
$$M \rar M^{\vee \vee}$$
is an isomorphism for any object $M$ (\cite{DeMi}).
These are true because $A$ is a field.
\end{proof}

\subsection{$L$-triviality}\label{phi_l}
Let $d$ be a positive integer and $F \subset E \subset L$ ring extensions where $F$, $E$ are fields
and $L = \prod_{l \in \mathbb{Z}/d} L_{l}$ is a finite product of fields.
For each $l$, we sometimes consider $L_{l}$ as a subset of $L$ in an obvious way.
Let $\sigma : L \rar L$ be a ring endomorphism.
We assume that the triple $(F, E, L)$ satisfies the following properties:
\begin{itemize}
\item $\sigma(E) \subset E$ and $\sigma(L_{l}) \subset L_{l + 1}$ for all $l$,
\item $F = E^{\sigma} = L^{\sigma}$,
\item $L$ is a separable extension over $E$.
\end{itemize}
Such a triple $(F, E, L)$ is called \textit{$\sigma$-admissible}.
See Lemma \ref{phi_v-triple} for our main example.
Another example can be found in \cite{Papa}.

Note that the separability of $L$ over $E$ is used to prove the smoothness of some algebraic groups
(see Theorem \ref{gal_smooth}), and not used in this section.

\begin{rema}\label{phi_papa}
In \cite{Papa}, the term $\sigma$-admissible triple is defined
only in the case where $L$ is a field and $\sigma$ is an isomorphism.
Thus our general setting urges us to argue with greater care than in \cite{Papa} at several points,
and hence we decided not to avoid repeating similar arguments.
\end{rema}

In this subsection, we consider $\vp$-modules over $(E, \sigma|_{E})$.
For any $\vp$-module $M$ over $E$, we set
$$V(M) := (L \otimes_{E} M)^{\vp}$$
where $\vp$ acts on $L \otimes_{E} M$ by $\sigma \otimes \vp$.
Then $V(M)$ is an $F$-vector space and $V$ forms a functor.
We have natural maps
$$\iota_{M} : L \otimes_{F} V(M) \rar L \otimes_{E} M,$$
$$\iota_{M,l} : L_{l} \otimes_{F} V(M) \rar L_{l} \otimes_{E} M \textrm{ for all } l.$$

\begin{lemm}\label{phi_l-ind}
Let $M$ be a $\vp$-module, and let $\mu_{1}, \dots, \mu_{m} \in V(M)$.
If $\mu_{1}, \dots, \mu_{m}$ are linearly independent over $F$,
then they are linearly independent over $L$ $($in $L \otimes_{E} M)$.
\end{lemm}

\begin{proof}
Assume that the lemma is not true.
Then there exist $m \geq 1$, $\mu_{1}, \dots, \mu_{m} \in V(M)$ and $f_{1}, \dots, f_{m} \in L$ such that,
$\mu_{1}, \dots, \mu_{m}$ are linearly independent over $F$,
$(f_{i})_{i} \neq 0$ and $\sum_{i} f_{i} \mu_{i} = 0$.
We may assume that $m$ is minimal among the integers which satisfy the above properties.
We also assume that $f_{1} = (a_{l})_{l} \in \prod_{l} L_{l}$ is non-zero.
Let $a_{l_0} \neq 0$.
Then there exists an element $f' \in L$ such that $f' f = e_{l_0}$,
where $e_{l_0} \in L$ is the element such that the $l_{0}$-th component is one and the other components are zero.
Then we have $\sum_{i} f' f_{i} \mu_{i} = 0$.
Therefore we may assume that $f_{1} = e_{l_0}$.
Then we have
$$0 = \sum_{j=0}^{d-1} \vp^{j} (\sum_{i=1}^{m} f_{i} \mu_{i}) = \sum_{i=1}^{m} \sum_{j=0}^{d-1} \vp^{j} (f_{i} \mu_{i})
= \sum_{i=1}^{m} (\sum_{j=0}^{d-1} \sigma^{j}(f_{i})) \mu_{i} = \mu_{1} + \sum_{i=2}^{m} (\sum_{j=0}^{d-1} \sigma^{j}(f_{i})) \mu_{i}.$$
Therefore we may assume that $f_{1} = 1$.
Then we have
$$0 = \vp(\sum_{i=1}^{m} f_{i} \mu_{i}) - \sum_{i=1}^{m} f_{i} \mu_{i}
= \sum_{i=1}^{m} (\sigma(f_{i}) - f_{i}) \mu_{i} = \sum_{i=2}^{m} (\sigma(f_{i}) - f_{i}) \mu_{i}.$$
By the minimality of $m$, we have $f_{i} \in L^{\sigma} = F$ for all $i$.
This contradicts the linear independence of $(\mu_{i})_{i}$ over $F$.
\end{proof}

\begin{coro}\label{phi_inj}
For any $\vp$-module $M$, the maps $\iota_{M}$ and $\iota_{M,l}$ are injective and we have $\dim_{F} V(M) \leq \dim_{E} M$.
\end{coro}

\begin{proof}
By Lemma \ref{phi_l-ind}, $\iota_{M}$ is injective.
It is clear that $\iota_{M}$ is injective if and only if $\iota_{M,l}$ are injective for all $l$.
Therefore $\iota_{M,l}$ is injective
and we have an inequality
$\dim_{F} V(M) = \dim_{L_l} (L_{l} \otimes_{F} V(M)) \leq \dim_{L_l} (L_{l} \otimes_{E} M) = \dim_{E} M$.
\end{proof}

\begin{defi}\label{phi_l-triv}
Let $M$ be a finite-dimensional $\vp$-module over $E$.
We say that $M$ is $L$\textit{-trivial} if the map $\iota_{M}$ is an isomorphism.
\end{defi}

We denote by $\phm_{E}^{L}$ the full subcategory of $\phm_{E}$ consisting of all $L$-trivial $\vp$-modules.
Let $M$ be a finite-dimensional $\vp$-module over $E$ and $\mbm \in \Mat_{r \times 1}(M)$ its $E$-basis.
Then there exists a matrix $\Phi \in \Mat_{r \times r}(E)$ such that $\vp \mbm = \Phi \mbm$.

\begin{prop}\label{phi_equiv}
The following conditions are equivalent:
\begin{description}
\item[$(1)$] $M$ is $L$-trivial,
\item[$(2)$] $\iota_{M,l}$ is an isomorphism for each $l$,
\item[$(3)$] $\iota_{M,l}$ is an isomorphism for some $l$,
\item[$(4)$] $\dim_{F} V(M) = \dim_{E} M$,
\item[$(5)$] there exists a matrix $\Psi \in \GL_{r}(L)$ such that $\sigma \Psi = \Phi \Psi$.
\end{description}
\end{prop}

\begin{proof}
It is clear that $(1) \iff (2) \Rightarrow (3) \Rightarrow (4)$.
Assume that the condition $(4)$ is true.
Then for each $l$, we have
$\dim_{L_l} (L_{l} \otimes_{F} V(M)) = \dim_{F} V(M) = \dim_{E} M = \dim_{L_l} (L_{l} \otimes_{E} M)$.
Therefore $\iota_{M,l}$ is an isomorphism.
This means that the condition $(4)$ implies the condition $(2)$.

Assume that the condition $(1)$ is true.
Let $\mbx$ be an $F$-basis of $V(M)$.
Since the natural map $\iota_{M} : L \otimes_{F} V(M) \rar L \otimes_{E} M$ is an isomorphism,
there exists a matrix $\Psi \in \GL_{r}(L)$ such that $\Psi \mbx = 1 \otimes \mbm$.
Then we have
$$(\sigma \Psi) \mbx = (\sigma \Psi) (\vp \mbx) = \vp(\Psi \mbx) = \vp(1 \otimes \mbm)
= 1 \otimes \vp \mbm = 1 \otimes \Phi \mbm = \Phi (1 \otimes \mbm) = \Phi \Psi \mbx.$$
By Lemma \ref{phi_l-ind}, we have $\sigma \Psi = \Phi \Psi$ and the condition $(5)$ is true.
Conversely, assume that the condition $(5)$ is true.
Then we have
$$\vp(\Psi^{-1}(1 \otimes \mbm)) = (\sigma \Psi)^{-1} (1 \otimes \vp \mbm)
= (\Phi \Psi)^{-1} (1 \otimes \Phi \mbm) = \Psi^{-1} (1 \otimes \mbm).$$
This means that $\Psi^{-1} (1 \otimes \mbm) \in \Mat_{r \times 1}(V(M))$.
Thus we have an inequality $\dim_{F} V(M) \geq \dim_{E} M$ and the condition $(4)$ is true.
\end{proof}

\begin{coro}\label{phi_l-et}
Let $M$ be a finite-dimensional $\vp$-module over $E$.
If $M$ is $L$-trivial then $M$ is \'etale.
\end{coro}

\begin{proof}
By Proposition \ref{phi_equiv}, there exists a matrix $\Psi \in \GL_{r}(L)$ such that $\sigma \Psi = \Phi \Psi$.
Since $\sigma$ is injective, we have $\det \Phi = \sigma(\det \Psi) \det \Psi^{-1} \neq 0$.
\end{proof}

Let $M$ be an $L$-trivial $\vp$-module over $E$,
$\mbm \in \Mat_{r \times 1}(M)$ an $E$-basis of $M$
and $\Phi \in \GL_{r}(E)$ a matrix such that $\vp \mbm = \Phi \mbm$.
By Proposition \ref{phi_equiv}, there exists a matrix $\Psi \in \GL_{r}(L)$ such that $\sigma \Psi = \Phi \Psi$.
\begin{defi}\label{phi_period}
The matrix $\Psi$ is called a \textit{period matrix} of $M$ in $L$ or \textit{fundamental matrix} of $\Phi$,
and the entries of $\Psi$ are called \textit{periods} of $M$ in $L$.
\end{defi}

Note that $\Psi' \in \GL_{r}(L)$ is another fundamental matrix of $\Phi$
if and only if $\Psi' = \Psi \delta$ for some $\delta \in \GL_{r}(F)$.
Indeed, if $\sigma(\Psi') = \Phi \Psi'$ then
$\sigma(\Psi^{-1} \Psi') = \sigma(\Psi)^{-1} \sigma(\Psi') = (\Phi \Psi)^{-1} (\Phi \Psi') = \Psi^{-1} \Psi'$,
hence $\Psi^{-1} \Psi' \in \GL_{r}(L^{\sigma}) = \GL_{r}(F)$, and the converse is clear.

\begin{prop}\label{phi_psi}
The period matrix $\Psi$ of $M$ is well-defined from $M$ as an element of $\GL_{r}(E) \backslash \GL_{r}(L) / \GL_{r}(F)$.
\end{prop}

\begin{proof}
Let $\mbm' \in \Mat_{r \times 1}(M)$ be another $E$-basis of $M$,
$\Phi' \in \GL_{r}(E)$ and $\Psi' \in \GL_{r}(L)$ matrices such that
$\vp \mbm' = \Phi' \mbm'$ and $\sigma \Psi' = \Phi' \Psi'$.
There exists a matrix $A \in \GL_{r}(E)$ which satisfies $\mbm' = A \mbm$.
Then we have
$\vp \mbm' = \vp (A \mbm) = \sigma(A) \vp \mbm = \sigma(A) \Phi \mbm = \sigma(A) \Phi A^{-1} \mbm'$.
Thus $\Phi' = \sigma(A) \Phi A^{-1}$.
We also have
$\sigma(A \Psi) = \sigma(A) \sigma(\Psi) = \sigma(A) \Phi \Psi = \Phi' (A \Psi)$.
Hence we conclude that
$\Psi' \in A \Psi \cdot \GL_{r}(F)$.
\end{proof}

\begin{prop}\label{phi_basis}
The set of entries of $\Psi^{-1} (1 \otimes \mbm)$ forms an $F$-basis of $V(M)$.
\end{prop}

\begin{proof}
By the proof of Proposition \ref{phi_equiv}, we have that $\Psi^{-1} (1 \otimes \mbm) \in \Mat_{r \times 1}(V(M))$.
Since $\dim_{F} V(M) = \dim_{E} M = r$, this is an $F$-basis of $V(M)$.
\end{proof}

\begin{prop}\label{phi_l-identity}
The $\vp$-module $(E, \sigma)$ is $L$-trivial.
\end{prop}

\begin{proof}
We have equalities $V(E) = (L \otimes_{E} E)^{\vp} = L^{\sigma} = F$.
Therefore $\dim_{F} V(E) = 1 = \dim_{E} E$.
\end{proof}

\begin{prop}\label{phi_l-op}
Let $M$ and $N$ be $L$-trivial $\vp$-modules.
Then $M \oplus N$, $M \otimes N$ and $\Hom(M, N)$ are also $L$-trivial.
\end{prop}

\begin{proof}
Let $\mbm \in \Mat_{r \times 1}(E)$ be an $E$-basis of $M$,
$\Phi_M \in \GL_{r}(E)$ the matrix such that $\vp \mbm = \Phi_M \mbm$
and $\Psi_M \in \GL_{r}(L)$ a matrix which satisfies $\sigma \Psi_M = \Phi_M \Psi_M$.
We also set $\mbn \in \Mat_{s \times 1}(E)$ an $E$-basis of $N$
and $\Phi_N \in \GL_{s}(E), \Psi_N \in \GL_{s}(L)$ matrices which satisfy
$\vp \mbn = \Phi_N \mbn$ and $\sigma \Psi_{N} = \Phi_{N} \Psi_{N}$.
We set
$$\mbm \oplus \mbn := \begin{bmatrix} \mbm \\ \mbn \\ \end{bmatrix},
\Phi_{M} \oplus \Phi_{N} := \begin{bmatrix} \Phi_{M} & 0 \\ 0 & \Phi_{N} \\ \end{bmatrix}
\textrm{ and } \Psi_{M} \oplus \Psi_{N} := \begin{bmatrix} \Psi_{M} & 0 \\ 0 & \Psi_{N} \\ \end{bmatrix}.$$
Then it is clear that $\mbm \oplus \mbn$ is an $E$-basis of $M \oplus N$,
$\vp (\mbm \oplus \mbn) = (\Phi_{M} \oplus \Phi_{N}) (\mbm \oplus \mbn)$
and $\sigma (\Psi_{M} \oplus \Psi_{N}) = (\Phi_{M} \oplus \Phi_{N}) (\Psi_{M} \oplus \Psi_{N})$.
Therefore $M \oplus N$ is $L$-trivial.

Set $\mbm \otimes \mbn$ to be an $E$-basis of $M \otimes N$ naturally obtained from $\mbm$ and $\mbn$.
Let $\Phi_{M} \otimes \Phi_{N}$ be the Kronecker product of $\Phi_{M}$ and $\Phi_{N}$,
and $\Psi_{M} \otimes \Psi_{N}$ be the Kronecker product of $\Psi_{M}$ and $\Psi_{N}$.
Then it is clear that $\vp (\mbm \otimes \mbn) = (\Phi_{M} \otimes \Phi_{N}) (\mbm \otimes \mbn)$
and $\sigma (\Psi_{M} \otimes \Psi_{N}) = (\Phi_{M} \otimes \Phi_{N}) (\Psi_{M} \otimes \Psi_{N})$.
Therefore $M \otimes N$ is $L$-trivial.

Let $\mbm^{\vee}$ be the dual basis of $\mbm$ for $M^{\vee}$.
Then we have equalities $\vp \mbm^{\vee} = (\Phi_{M}^{-1})^{\mathrm{tr}} \mbm^{\vee}$
and $\sigma (\Psi_{M}^{-1})^{\mathrm{tr}} = (\Phi_{M}^{-1})^{\mathrm{tr}} (\Psi_{M}^{-1})^{\mathrm{tr}}$,
where $A^{\mathrm{tr}}$ is the transpose of a matrix $A$.
Therefore $M^{\vee}$ is $L$-trivial.
Since $\epm_{E}$ is a rigid tensor category, we have an isomorphism $M^{\vee} \otimes N \cong \Hom(M, N)$.
Therefore $\Hom(M, N)$ is $L$-trivial.
\end{proof}

\begin{prop}\label{phi_l-subq}
Let $0 \rar M' \rar M \rar M'' \rar 0$ be an exact sequence in $\phm_{E}$.
If $M$ is $L$-trivial, then $M'$ and $M''$ are also $L$-trivial.
\end{prop}

\begin{proof}
Let $\kappa : L \otimes_{F} V(M) \rar L \otimes_{F} V(M'')$ be the natural map
and $\iota'' : \im(\kappa) \rar L \otimes_{E} M''$ be the restriction of the map $\iota_{M''}$.
Then we have the commutative diagram
\[\xymatrix{
0 \ar[r] & L \otimes_{F} V(M') \ar[r] \ar[d]_{\iota_{M'}} & L \otimes_{F} V(M) \ar[r] \ar[d]_{\iota_{M}} & \im(\kappa) \ar[r] \ar[d]_{\iota''} & 0 \\
0 \ar[r] & L \otimes_{E} M' \ar[r] & L \otimes_{E} M \ar[r] & L \otimes_{E} M'' \ar[r] & 0, \\
}\]
where the rows are exact, $\iota_{M}$ is an isomorphism and $\iota_{M'}, \iota''$ are injective.
Then we have that $\iota_{M'}$ and $\iota''$ are isomorphism by a diagram chasing.
Hence $M'$ and $M''$ are $L$-trivial.
\end{proof}

\begin{prop}\label{phi_l-rigid}
The category $\phm_{E}^{L}$ is a rigid abelian $F$-linear tensor category.
\end{prop}

\begin{proof}
By Proposition \ref{phi_rigid} and Corollary \ref{phi_l-et},
it is enough to show that the category $\phm_{E}^{L}$ is closed under
direct sum, subquotient, tensor product and internal hom,
and has an identity object for the tensor product.
By Propositions \ref{phi_l-identity}, \ref{phi_l-op} and \ref{phi_l-subq}, these are true.
\end{proof}

By Corollary \ref{phi_inj}, we can consider $V$ as a functor from $\phm_{E}^{L}$
to the category of finite-dimensional $F$-vector spaces $\vsp(F)$.

\begin{prop}\label{phi_exact}
The functor $V : \phm_{E}^{L} \rar \vsp(F)$ is $F$-linear and exact.
\end{prop}

\begin{proof}
It is clear that $V$ is $F$-linear.
Let $0 \rar M' \rar M \rar M'' \rar 0$ be an exact sequence in $\phm_{E}^{L}$.
It is clear that $0 \rar V(M') \rar V(M) \rar V(M)$ is exact.
This means that $\dim_{F} V(M) \leq \dim_{F} V(M') + \dim_{F} V(M'')$.
On the other hand, we have
$\dim_{F} V(M) = \dim_{E} M = \dim_{E} M' + \dim_{E} M'' \geq \dim_{F} V(M') + \dim_{F} V(M'')$.
\end{proof}

\begin{prop}\label{phi_faithful}
The functor $V : \phm_{E}^{L} \rar \vsp(F)$ is faithful.
\end{prop}

\begin{proof}
Let $M$ and $N$ be $L$-trivial $\vp$-modules and $\phi : M \rar N$ a morphism in $\phm_{E}$.
Then we have an exact sequence
\[\xymatrix{
0 \ar[r] & V(\ker \phi) \ar[r] & V(M) \ar[r]^{V(\phi)} & V(N). \\
}\]
Therefore if $V(\phi) = 0$ then $V(\ker \phi) = V(M)$.
Since $M$ is $L$-trivial,
we have an inequality $\dim_{E} \ker \phi \geq \dim_{F} V(\ker \phi) = \dim_{F} V(M) = \dim_{E} M$.
This means that $\ker \phi = M$ and $\phi = 0$.
\end{proof}

\begin{prop}\label{phi_t-func}
Let $M$ and $N$ be $L$-trivial $\vp$-modules.
Then there exists a natural isomorphism $V(M) \otimes_{F} V(N) \rar V(M \otimes N)$.
The functor $V : \phm_{E}^{L} \rar \vsp(F)$ is a tensor functor with respect to this isomorphism.
\end{prop}

\begin{proof}
It is clear that there exists a natural isomorphism
$(L \otimes_{E} M) \otimes_{L} (L \otimes_{E} N) \cong L \otimes_{E} (M \otimes N)$ which preserves $\vp$-actions.
By Lemma \ref{phi_l-ind}, the natural map
$V(M) \otimes_{F} V(N) \rar (L \otimes_{E} M) \otimes_{L} (L \otimes_{E} N)$ is injective.
Therefore we have a natural injection
$$V(M) \otimes_{F} V(N) \hookrightarrow ((L \otimes_{E} M) \otimes_{L} (L \otimes_{E} N))^{\vp}
\cong (L \otimes_{E} (M \otimes N))^{\vp} = V(M \otimes N).$$
Since $\dim_{F} (V(M) \otimes_{F} V(N)) = \dim_{F} V(M \otimes N)$, this map is a bijection.
It is clear that this isomorphism is compatible with the associativity and the commutativity of tensor functors.
It is also clear that $V(E) = F$.
Thus the functor $V$ is a tensor functor (\cite{DeMi}, Definition 1.8).
\end{proof}

Recall that a \textit{neutral Tannakian category} over a field $k$ is
a rigid abelian $k$-linear tensor category $\mathcal{C}$ for which
$k \xrightarrow{\sim} \End(\mathbf{1})$
and there exists an exact faithful $k$-linear tensor functor $\omega : \mathcal{C} \rar \vsp(k)$,
where $\mathbf{1}$ is the unit object of $\mathcal{C}$ (\cite{DeMi}, Definition 2.19).
Any such functor $\omega$ is said to be a \textit{fiber functor} for $\mathcal{C}$.

\begin{theo}\label{phi_tan}
The category $\phm_{E}^{L}$ is a neutral Tannakian category over $F$.
The functor $V : \phm_{E}^{L} \rar \vsp(F)$ is a fiber functor for $\phm_{E}^{L}$.
\end{theo}

\begin{proof}
By Proposition \ref{phi_endo}, we have $\End_{\vp}(E) \cong E^{\sigma} = F$.
By Proposition \ref{phi_l-rigid}, the category $\phm_{E}^{L}$ is a rigid abelian $F$-linear tensor category.
By Propositions \ref{phi_exact}, \ref{phi_faithful} and \ref{phi_t-func},
the functor $V : \phm_{E}^{L} \rar \vsp(F)$ is an exact faithful $F$-linear tensor functor.
\end{proof}

Let $M$ be an $L$-trivial $\vp$-module over $E$.
We set $\TM$ to be the Tannakian subcategory of $\phm_{E}^{L}$ generated by $M$,
and $V_{M} : \TM \rar \vsp(F)$ the restriction of the functor $V$.
We denote by $\Gamma_{M}$ the Tannakian Galois group of $(\TM, V_{M})$.
For any $F$-algebra $R$, we define the functor $V_{M}^{R} : \TM \rar \Mod(R)$ by $N \mapsto R \otimes_{F} V(N)$,
where $\Mod(R)$ is the category of finitely generated $R$-modules.
Then by the definition of $\Gamma_{M}$, we have
$$\Gamma_{M}(R) = \Aut^{\otimes}(V_{M}^{R})$$
where $\Aut^{\otimes}(V_{M}^{R})$ is the group of invertible natural transformations from $V_{M}^{R}$ to itself
which is compatible with the tensor products.
Therefore we have a canonical injection $\Gamma_{M}(R) \hookrightarrow \GL(R \otimes_{F} V(M))$.

\subsection{$v$-adic case}\label{phi_v}
Let $t$ be a variable and $v \in \Fq[t]$ a fixed monic irreducible polynomial of degree $d$.
For any field $k$ containing $\Fq$, we set $k[t]_{v} := \varprojlim(k[t]/v^n)$
and $k(t)_{v} := \Fq(t) \otimes_{\Fq[t]} k[t]_{v}$.

Let $\sigma$ be the ring endomorphism of $k[t]$
$$\sum a_i t^i \mapsto \sum a_i^q t^i.$$
Then $\sigma$ naturally extends to an endomorphism of $k(t)_{v}$, also denoted by $\sigma$.
Let $k'$ be a splitting field of $v$ over $k$ in $\bar{k}$,
and we factorize $v=\prod_{l \in \mathbb{Z}/d} (t-\lambda_{l})$ in $k'[t]$
with $\lambda_{l}^{q} = \lambda_{l + 1}$ for all $l \in \mathbb{Z}/d$.
Then we have $k'(t)_{v} = \prod_{l \in \mathbb{Z}/d} k'(\!(t-\lambda_{l})\!)$,
and for any $a=(\sum_{i} a_{l, i} (t-\lambda_{l})^{i})_{l} \in k'(t)_{v}$,
$$\sigma(a) = (\sum a_{l - 1, i}^{q} (t-\lambda_{l})^{i})_{l}.$$

\begin{lemm}\label{phi_fix}
For any field $k$ containing $\Fq$, we have $(k(t)_{v})^{\sigma} = \ftv$.
\end{lemm}

\begin{proof}
Clearly, $\ftv = (\ftv)^{\sigma} \subset (k(t)_{v})^{\sigma} \subset (k'(t)_{v})^{\sigma}$.
By the explicit description of the $\sigma$-action as above, we have
$(k'(t)_{v})^{\sigma} =
\{(\sum a_{l, i} (t-\lambda_{l})^{i})_{l} \in \mathbb{F}_{q^d}(t)_{v} | a_{l, i}^{q} = a_{l + 1, i}$ for all $l$ and $i \}$.
This set is isomorphic to $\mathbb{F}_{q^d}(\!(t-\lambda_{l})\!)$ via the $l$-th projection for any $l$.
On the other hand, we have $\ftv \cong \mathbb{F}_{q^d}(\!(t-\lambda_{l})\!)$.
Thus the above inclusions are all equalities.
\end{proof}

Fix a field $K$ containing $\Fq$ and assume that $K \cap \overline{\Fq} = \Fq$.
Note that if $\Fq$ is not algebraically closed in $K$,
then $\ktv$ may not be a field and the situation becomes more complicated.
Thus in this paper, we always assume that $K \cap \overline{\Fq} = \Fq$.

\begin{lemm}\label{phi_v-triple}
The triple $(\ftv, \ktv, \kstv)$ is $\sigma$-admissible.
\end{lemm}

\begin{proof}
Since $v$ is irreducible in $K[t]$, $\ktv$ is a field.
By Lemma \ref{phi_fix}, we have $\ftv = (\ktv)^{\sigma} = (\kstv)^{\sigma}$.
We need to check the separability.
Fix an $l$. We need to show that $\kstl/K(t)_{v}$ is a separable field extension.
It is clear that $\ktv = K'(\!(t-\lambda_{l})\!)$ where $K' = K(\lambda_{l})$.
On the other hand, $\kstl/K'(\!(t-\lambda_{l})\!)$ is separable since $K^{\mathrm{sep}}/K'$ is separable
(\cite{Mats}, Exercise 26.2).
\end{proof}

Let $G_{K} := \Gal(K^{\mathrm{sep}}/K)$ be the absolute Galois group of $K$.
Then $G_{K}$ acts on $K^{\mathrm{sep}}[t]$ in an obvious way.
This action naturally extends to an action on $\kstv$.
For each $\tau \in G_{K}$ and $a = (\sum_{i} a_{l,i} (t-\lambda_{l})^{i})_{l} \in \prod_{l} \kstl$, we have
$$\tau a = (\sum_{i} \tau a_{l+n,i} (t-\lambda_{l})^{i})_{l},$$
where $n \in \mathbb{Z}/d$ is an element such that $\tau|_{\mathbb{F}_{q^d}} = \sigma|_{\mathbb{F}_{q^d}}^{-n}$.
It is clear that this action is compatible with $\sigma$.

From now on, we consider $\vp$-modules over the $\sigma$-admissible triple $(\ftv,$ $\ktv,$ $\kstv)$.
Let $M$ be an \'etale $\vp$-module over $\ktv$.
The Galois group $G_{K}$ acts on $\kstv \otimes M$ continuously by $\tau \otimes \id$ for each $\tau \in G_{K}$.
Since this action is compatible with $\sigma$, the $\ftv$-subspace $V(M)$ is $G_{K}$-stable.
We denote by $V_{K}(M)$ this Galois representation.
Conversely for any object $V$ of $\Rep(G_{K}, \ftv)$, we set
$$D(V) := (\kstv \otimes_{\ftv} V)^{G_K},$$
where $G_{K}$ acts on $\kstv \otimes_{\ftv} V$ by $\tau \otimes \tau$ for $\tau \in G_{K}$.
Then we can define a $\vp$-action on $D(V)$ by $\sigma \otimes \id$.

Let $M_{0}$ be an \'etale $\vp$-module over $K[t]_{v}$.
Then we can define an $\Fq[t]_{v}$-representation of $G_{K}$
$$V_{0}(M_{0}) := (K^{\mathrm{sep}}[t]_{v} \otimes_{K[t]_{v}} M_{0})^{\vp},$$
where $\vp$ acts on $K^{\mathrm{sep}}[t]_{v} \otimes_{K[t]_{v}} M_{0}$ by $\sigma \otimes \vp$
and $G_{K}$ acts on $V_{0}(M_{0})$ by $\tau \otimes \id$ for $\tau \in G_{K}$.
Conversely for any object $T$ of $\Rep(G_{K}, \Fq[t]_{v})$, we set
$$D_{0}(T) := (K^{\mathrm{sep}}[t]_{v} \otimes_{\Fq[t]_{v}} T)^{G_K},$$
where $G_{K}$ acts on $K^{\mathrm{sep}}[t]_{v} \otimes_{\Fq[t]_{v}} T$ by $\tau \otimes \tau$ for $\tau \in G_{K}$.
Then we can define a $\vp$-action on $D_{0}(T)$ by $\sigma \otimes \id$.

\begin{theo}[\cite{Goss}, Appendix]\label{phi_int-equiv}
$(1)$ For any \'etale $\vp$-module $M_{0}$ over $K[t]_{v}$, the natural map
$$K^{\mathrm{sep}}[t]_{v} \otimes_{\Fq[t]_{v}} V_{0}(M_{0}) \rar K^{\mathrm{sep}}[t]_{v} \otimes_{K[t]_{v}} M_{0}$$
is bijective.

$(2)$ For any $\Fq[t]_{v}[G_{K}]$-module $T$ of finite type over $\Fq[t]_{v}$, the natural map
$$K^{\mathrm{sep}}[t]_{v} \otimes_{K[t]_{v}} D_{0}(T) \rar K^{\mathrm{sep}}[t]_{v} \otimes_{\Fq[t]_{v}} T$$
is bijective and the $\vp$-module $D_{0}(T)$ is \'etale.

$(3)$ The functor $V_{0} : \epm_{K[t]_{v}} \rar \Rep(G_{K}, \Fq[t]_{v})$ is a tensor equivalence,
with a quasi-inverse $D_{0} : \Rep(G_{K}, \Fq[t]_{v}) \rar \epm_{\Fq[t]_{v}}$.
\end{theo}

For any $\vp$-module $M_{0}$ over $K[t]_{v}$,
we can define a $\vp$-action on $\ktv \otimes_{K[t]_{v}} M_{0}$ by $\sigma \otimes \vp$.
\begin{theo}\label{phi_main-equiv}
$(1)$ A $\vp$-module $M$ over $\ktv$ is $\kstv$-trivial if and only if
there exists a subspace $M_{0}$ of $M$ which is an \'etale $\vp$-module over $K[t]_{v}$
such that $M = \ktv \otimes_{K[t]_{v}} M_{0}$.

$(2)$ For any object $V$ in $\Rep(G_{K}, \ftv)$, the $\vp$-module $D(V)$ is $\kstv$-trivial.

$(3)$ The functor $V : \phm_{\ktv}^{\kstv} \rar \Rep(G_{K}, \ftv)$ is a tensor equivalence,
with a quasi-inverse $D : \Rep(G_{K}, \ftv) \rar \phm_{\ktv}^{\kstv}$.
\end{theo}

\begin{proof}
Let $M$ be a $\vp$-module over $\ktv$ such that there exists a subspace $M_{0}$
which is an \'etale $\vp$-module over $K[t]_{v}$ and $M = \ktv \otimes_{K[t]_{v}} M_{0}$.
Then by Theorem \ref{phi_int-equiv} $(1)$, we have an isomorphism
$K^{\mathrm{sep}}[t]_{v} \otimes_{\Fq[t]_{v}} V_{0}(M_{0}) \cong K^{\mathrm{sep}}[t]_{v} \otimes_{K[t]_{v}} M_{0}$.
By tensoring $\kstv$ to the both sides of this isomorphism, we conclude that $M$ is $\kstv$-trivial.

Let $V$ be an object in $\Rep(G_{K}, \ftv)$.
Then there exists a $G_{K}$-stable $\Fq[t]_{v}$-lattice $T$ for $V$.
It is clear that $D_{0}(T)$ is free over $K[t]_{v}$ and $D(V) = \ktv \otimes_{K[t]_{v}} D_{0}(T)$.
Thus $D(V)$ is $\kstv$-trivial from the above argument and this proves $(2)$.
By Theorem \ref{phi_int-equiv} $(2)$, we have an isomorphism
$\kstv \otimes_{\ktv} D(V) \cong \kstv \otimes_{\ftv} V$.
By taking the $\vp$-fixed parts of the both sides of this isomorphism, we have an isomorphism $V_{K}(D(V)) \cong V$.

Let $M$ be a $\kstv$-trivial $\vp$-module over $\ktv$.
Then we have an isomorphism $\kstv \otimes_{\ftv} V_{K}(M) \cong \kstv \otimes_{\ktv} M$.
By taking the $G_{K}$-fixed parts of the both sides of this isomorphism, we have an isomorphism $D(V_{K}(M)) \cong M$,
and this proves $(3)$.
Therefore $M$ comes from \'etale $\vp$-module over $K[t]_{v}$ and this proves $(1)$.
\end{proof}

\section{Frobenius equations}\label{sec_gal}

Throughout this section, we fix a $\sigma$-admissible triple $(F, E, L)$.

\begin{exam}
The case $(F, E, L) = (\mathbb{F}_{q}(t)_{v}, K(t)_{v}, K^{\mathrm{sep}}(t)_{v})$
is our main example of a $\sigma$-admissible triple,
where the notation and the $\sigma$-action are as in Subsection \ref{phi_v}.
\end{exam}

\begin{exam}\label{gal_sigma-inv}
Let $(F, E, L) = (\mathbb{F}_{q}(t)_{v}, K^{\mathrm{rad}}(t)_{v}, \bar{K}(t)_{v})$
where $K^{\mathrm{rad}} := \cup_{n} K^{1/q^{n}}$, the maximal radical extension of $K$ in $\bar{K}$.
The automorphism of $K^{\mathrm{rad}}(t)$
$$\sum_{i} a_{i} t^{i} \mapsto \sum_{i} a_{i}^{1/q} t^{i}$$
is naturally extends to an automorphism of $\bar{K}(t)_{v}$.
We define $\sigma$ to be this action.
Then $(\mathbb{F}_{q}(t)_{v}, K^{\mathrm{rad}}(t)_{v}, \bar{K}(t)_{v})$ is a $\sigma$-admissible triple.
Note that, in this case we need to put $L_{l} = \bar{K}(\!(t-\lambda_{-l})\!)$.
Note also that we do not use this type in this paper.
However, the $\sigma$-action of this type is used in \cite{Papa} and \cite{ChYu}.
\end{exam}

\subsection{The group $\Gamma$}\label{gal_gal-gp}

Let $r$ be a positive integer.
Fix matrices $\Phi = (\Phi_{ij}) \in \GL_{r}(E)$ and $\Psi = (\Psi_{ij})_{i,j} \in \GL_{r}(L)$
such that $\Psi$ is a fundamental matrix for $\Phi$.
Thus we have an equation $$\sigma(\Psi) = \Phi \Psi.$$
This means that the matrices $\Phi$ and $\Psi$ come from an $L$-trivial $\vp$-module over $E$.
Since $L = \prod_{l} L_{l}$, we can write $\Psi_{ij} = (\Psi_{ijl})_{l}$ for each $i$ and $j$.
We set $\Psi_{l} := (\Psi_{ijl})_{i,j} \in \GL_{r}(L_{l})$.
Then we have $\sigma(\Psi_{l}) = \Phi \Psi_{l+1}$ for all $l$.

Let $X := (X_{ij})$ be an $r \times r$ matrix of independent variables $X_{ij}$, and set $\Delta := \det(X)$.
We set $\EX := E[X_{11}, X_{12}, \dots , X_{rr}, \Delta^{-1}]$.
Similarly $E[\Psi, \Delta(\Psi)^{-1}]$ and $E[\Psi_{l}, \Delta(\Psi_{l})^{-1}]$ are defined.
We define $E$-algebra homomorphisms
$\nu : \EX \rar L ; \ X_{ij} \mapsto \Psi_{ij}$ and
$\nu_{l} : \EX \rar L_{l} ; \ X_{ij} \mapsto \Psi_{ijl}$.
Set $\mfp := \ker \nu$, $\Sigma := E[\Psi, \Delta(\Psi)^{-1}] \cong \EX/\mfp$, $Z := \Spec \Sigma$,
$\mfp_{l} := \ker \nu_{l}$, $\Sigma_{l} := E[\Psi_{l}, \Delta(\Psi_{l})^{-1}] \cong \EX/\mfp_{l}$
and $Z_{l} := \Spec \Sigma_{l}$.
Then $Z_{l}$ are closed subschemes of $Z$ and $Z = \cup_{l} Z_{l}$.
Let $\La := \Frac(\Sigma)$ and $\Lal := \Frac(\Sl)$, the total rings of fractions.

Set $\Psi_{1} := (\Psi_{ij} \otimes 1)_{i,j}$, $\Psi_{2} := (1 \otimes \Psi_{ij})_{i,j}$ and
$\widetilde{\Psi} = (\widetilde{\Psi}_{ij})_{i,j} := \Psi_{1}^{-1} \Psi_{2}$ in $\GL_{r}(L \otimes_{E} L)$.
Since $L \otimes_{E} L = \prod_{l,m} L_{l} \otimes_{E} L_{m}$,
we can write $\widetilde{\Psi}_{ij} = (\widetilde{\Psi}_{ijlm})_{l,m}$
with $\widetilde{\Psi}_{ijlm} \in L_{l} \otimes_{E} L_{m}$ for each $i$ and $j$.
We define $F$-algebra homomorphisms
$\mu : \FX \rar L \otimes_{E} L ; \ X_{ij} \mapsto \widetilde{\Psi}_{ij}$ and
$\mu_{lm} : \FX \rar L_{l} \otimes_{E} L_{m} ; \ X_{ij} \mapsto \widetilde{\Psi}_{ijlm}$.
Set $\mfq := \ker \mu$, $\Gamma := \Spec \FX/\mfq$,
$\mfq_{lm} := \ker \mu_{lm}$ and $\Gamma_{lm} := \Spec \FX/\mfq_{lm}$.
Then $\Gamma_{lm}$ are closed subschemes of $\Gamma$ and $\Gamma = \cup_{l,m} \Gamma_{lm}$.
By the next lemma, we can set $\mfq_{m} := \mfq_{0,m} = \mfq_{1,m+1} = \cdots$
and $\Gamma_{m} := \Gamma_{0,m} = \Gamma_{1,m+1} = \cdots$.

\begin{lemm}
For any $l, m \in \mathbb{Z}/d$, we have $\mfq_{lm} = \mfq_{l+1,m+1} = \mfq_{l+2,m+2} = \cdots$.
\end{lemm}

\begin{proof}
Let $\tilde{L}$ be the inductive limit of the inductive system
$L \rar L \rar L \rar \cdots$, where the transition maps are $\sigma$.
Then $L$ is a subring of $\tilde{L}$ and $\sigma$ is naturally extends to an automorphism of $\tilde{L}$.
We can define a $\sigma$-action on $\tilde{L} \otimes_{E} \tilde{L}$ by $\sigma \otimes \sigma$.
This is an isomorphism and $L \otimes_{E} L$ is stable under this action.
Thus we obtain an injective endomorphism $\sigma$ of $L \otimes_{E} L$.
It is clear that $\sigma(L_{l} \otimes_{E} L_{m}) \subset L_{l+1} \otimes_{E} L_{m+1}$.

Write $\Psi_1 = (\Psi_{1,lm})_{l,m}$ and $\Psi_2 = (\Psi_{2,lm})_{l,m}$ with $\Psi_{i,lm} \in \GL_{r}(L_{l} \otimes_{E} L_{m})$,
and set $\widetilde{\Psi}^{(lm)} := (\widetilde{\Psi}_{ijlm})_{i,j} \in \GL_{r}(L_{l} \otimes_{E} L_{m})$ for each $l$ and $m$.
Then we obtain the equality
$\sigma(\widetilde{\Psi}^{(lm)}) = \sigma(\Psi_{1,lm})^{-1} \sigma(\Psi_{2,lm})
= (\Phi \Psi_{1,l+1,m+1})^{-1} (\Phi \Psi_{2,l+1,m+1}) = \widetilde{\Psi}^{(l+1,m+1)}$.
For any $h(X) \in \FX$, we have
$h(\widetilde{\Psi}^{(lm)}) = 0$ if and only if $h(\widetilde{\Psi}^{(l+1,m+1)}) = 0$
since $\sigma(h(\widetilde{\Psi}^{(lm)})) = h(\widetilde{\Psi}^{(l+1,m+1)})$
and $\sigma$ is injective on $L \otimes_{E} L$.
This proves the lemma.
\end{proof}

For any $h(X) \in \LX$, we denote by $h^{\sigma}(X)$
the polynomial obtained by applying $\sigma$ to the coefficients of $h(X)$.
We define two endomorphisms
$$\sigma_{0} : \LX \rar \LX ; \ h(X) \mapsto h^{\sigma}(X),$$
$$\sigma_{1} : \LX \rar \LX ; \ h(X) \mapsto h^{\sigma}(\Phi X).$$
Then $\sigma_{0}(\LlX) \subset L_{l+1}\X$ and $\sigma_{1}(\LlX) \subset L_{l+1}\X$.

\begin{lemm}\label{gal_sigma}
We have $\sigma_{1} \mfp \subset \mfp$, $\sigma_{1} \mfp_{l} \subset \mfp_{l+1}$,
$\sigma_{0} \mfq = \mfq$, $\sigma_{0} \mfq_{m} = \mfq_{m}$,
$\sigma \nu = \nu \sigma_{1}|_{\EX}$ and $\sigma \nu_{l} = \nu_{l+1} \sigma_{1}|_{\EX}$
for each $l$ and $m$.
\end{lemm}

\begin{proof}
For any $h(X) \in \EX$, we have
$\nu_{l+1}(\sigma_{1}(h(X))) = (\sigma_{1} h)(\Psi_{l+1}) = h^{\sigma}(\Phi \Psi_{l+1})
= h^{\sigma}(\sigma \Psi_{l}) = \sigma(h(\Psi_{l})) = \sigma(\nu_{l}(h(X)))$.
If $h \in \mfp_{l}$, then $(\sigma_{1}h)(\Psi_{l+1}) = \sigma(h(\Psi_{l})) = 0$,
and hence $\sigma_{1}h \in \mfp_{l+1}$.
Since $\mfq_{l} \subset \FX$ and $\sigma_{0}|_{\FX} = \id$,
we have $\sigma_{0}\mfq_{l} = \mfq_{l}$.
The other assertions are proved similarly.
\end{proof}

For any ring homomorphism $R \rar S$ and any ideal $\mfa \subset R\X$,
we set $\mfa_S := \mfa \cdot S\X$, the extension ideal of $\mfa$.

\begin{lemm}\label{gal_bij1}
There exists a bijection between the set of ideals of $\FX$ and the set of ideals of $\LX$ which are $\sigma_{0}$-stable,
via the extension and the restriction of ideals.
\end{lemm}

\begin{proof}
For any ideal $\mfa \subset \FX$,
it is clear that $\sigma_{0} \mfa_{L} \subset \mfa_{L}$.
Because of the faithfully flatness of the inclusion $\FX \hookrightarrow \LX$,
we have $\mfa = \mfa_{L} \cap \FX$.

Conversely, we take any ideal $\mfb \subset \LX$ with $\sigma_{0} \mfb \subset \mfb$,
and set $\mfa := \mfb \cap \FX$.
It is clear that $\mfb \supset \mfa_L$;
thus we need to show that the converse inclusion $\mfb \subset \mfa_L$.

Take an $F$-basis $(g_{i})_{i \in I}$ of $\FX$.
Then this is an $L$-basis of $\LX$.
For each $h = \sum_{i} b_{i} g_{i} \in \LX$,
we set $\supp(h) := \{i \in I | b_{i} \neq 0\}$ and $l(h) := \# \supp(h)$.
We take $h \in \mfb$ and show that $h \in \mfa_{L}$ by induction on $l(h)$.
If $l(h) = 0$, then $h = 0 \in \mfa_{L}$.
Now suppose that $l(h) > 0$, and assume that
if $\tilde{h} \in \mfb$ and $l(\tilde{h}) < l(h)$ then $\tilde{h} \in \mfa_{L}$.
Let $e_{l} \in L$ be the element such that the $l$-th component is one and the other components are all zero.
Then it is clear that $\sigma e_{l} = e_{l+1}$.
We write $h = \sum_{i} b_{i} g_{i}$ and take $i_1$ such that $b_{i_1} \neq 0$.
Take $l_0$ such that the $l_0$-th component of $b_{i_1}$ is non-zero.
Then there exists an element $b' \in L$ such that $b'b_{i_1} = e_{l_0}$.
Since $\mfb$ is an ideal and $\sigma_{0}$-stable, we have
$$\mfb \ni \sum_{j=0}^{d-1} \sigma_{0}^{j}(b'h) = \sum_{j=0}^{d-1} \sigma_{0}^{j}(b'\sum_{i} b_{i} g_{i})
= \sum_{i} \sum_{j=0}^{d-1} \sigma^{j}(b'b_{i}) g_{i} =: \sum_{i} c_{i} g_{i} =: h',$$
$c_{i_1} = \sum_{j} \sigma^{j}(b'b_{i_1}) = 1$ and $\supp(h') \subset \supp(h)$.
Therefore $h - b_{i_1} h' \in \mfb$ and $l(h - b_{i_1} h') < l(h)$.
By induction hypothesis, we have $h - b_{i_1} h' \in \mfa_{L}$.
Hence it is enough to show that $h' \in \mfa_{L}$.
If $c_{i} \in F$ for all $i$,
then $h' \in \mfb \cap \FX = \mfa \subset \mfa_{L}$.
If $\# L_{l} \leq 3$ for some (hence for all) $l$,
we can write the $\sigma$ action on $L$ by $(x_{l})_{l} \mapsto (x_{l-1})_{l}$.
Hence $c_{i} \in F$ for all $i$.
Thus we assume that $c_{i_2} \in L \smallsetminus F$ for some $i_2$ and $\# L_{l} \geq 4$ for all $l$.

We claim that, we can construct an element $\bar{h} = \sum_{i} a_{i} g_{i} \in \mfa_{L}$
which has the properties that $\supp(\bar{h}) \subset \supp(h')$ and $a_{i_1} = 1$.
We first show that the claim implies $h' \in \mfa_{L}$.
Since $h' - \bar{h} \in \mfb$ and $l(h' - \bar{h}) < l(h') \leq l(h)$,
we have $h' - \bar{h} \in \mfa_{L}$ by induction hypothesis.
Thus we have $h' \in \mfa_{L}$ since $\bar{h} \in \mfa_{L}$.

Now we prove the claim.
First, we construct an element $\bar{h} = \sum_{i} a_{i} g_{i} \in \mfb$
which has the properties that $\supp(\bar{h}) \subset \supp(h')$, $a_{i_1} = 1$,
$a_{i_2} \in L^{\times}$ and $\sigma(a_{i_2}^{-1})-a_{i_2}^{-1} \in L^{\times}$.
If $d = 1$, then we can take $\bar{h} = h'$ since $L$ is a field and $\sigma(c_{i_2}^{-1}) - c_{i_2}^{-1} \neq 0$.
Thus we suppose that $d \geq 2$.
Since $c_{i_2} = (c_{i_2, l})_{l} \not\in F$, there exists an $l_{0}$ such that $\sigma c_{i_2, l_0 - 1} \neq c_{i_2, l_0}$.
Thus there is an element $c' \in L$ such that $c'(\sigma c_{i_2} - c_{i_2}) = e_{l_0}$.
We set
$$h'' := \sum_{i} \alpha_{i} g_{i} := \sum_{i} \sum_{j=0}^{d-1} \sigma_{0}^{j}(c'(\sigma c_{i} - c_{i})) g_{i}
= \sum_{j=0}^{d-1} \sigma_{0}^{j}(c'(\sigma_{0} h' - h')) \in \mfb.$$
Then we have $\alpha_{i_1} = 0$, $\alpha_{i_2} = 1$ and $\supp(h'') \subset \supp(h')$.
For $f = (f_{l})_{l} \in L$, consider the element
$\bar{h} := h' - f h'' = g_{i_1} + (c_{i_2, l} - f_{l})_{l} g_{i_2} + \cdots \in \mfb$.
For any $x = (x_{l})_{l} \in L^{\times}$, $\sigma x^{-1} - x^{-1} \in L^{\times}$
if and only if $\sigma x_{l} \neq x_{l+1}$ for all $l$.
Therefore, it is enough to take an element $f$ such that
$c_{i_2, l} \neq f_{l}$ and $\sigma(c_{i_2, l-1} - f_{l-1}) \neq c_{i_2, l} - f_{l}$ for all $l$.
Since $\# L_{l} \geq 4$, we can take $(f_{l})_{l}$ inductively so that
$f_{1} \in L_{1} \smallsetminus \{c_{i_2, 1}\}$,
$f_{l} \in L_{l} \smallsetminus \{c_{i_2, l}, c_{i_2, l} - \sigma(c_{i_2, l-1} - f_{l-1})\}$ for $2 \leq l < d$ and
$f_{d} \in L_{d} \smallsetminus (\{c_{i_2, d}, c_{i_2, d} - \sigma(c_{i_2, d-1} - f_{d-1})\} \cup \sigma^{-1}(\sigma(c_{i_2, d}) - c_{i_2, 1} + f_{1}))$.
Then such $(f_{l})_{l}$ satisfies the above properties.
Next, we show that $\bar{h} \in \mfa_{L}$.
Since $\sigma_{0} \bar{h} - \bar{h} \in \mfb$ and $l(\sigma_{0} \bar{h} - \bar{h}) < l(\bar{h}) \leq l(h)$,
we have $\sigma_{0} \bar{h} - \bar{h} \in \mfa_{L}$ by induction hypothesis.
Similarly, we can show that $\sigma_{0}(a_{i_2}^{-1} \bar{h}) - a_{i_2}^{-1} \bar{h} \in \mfa_{L}$.
Therefore we have
$(\sigma(a_{i_2}^{-1}) - a_{i_2}^{-1}) \bar{h}
= (\sigma_{0}(a_{i_2}^{-1} \bar{h}) - a_{i_2}^{-1} \bar{h}) - \sigma(a_{i_2}^{-1})(\sigma_{0} \bar{h} - \bar{h}) \in \mfa_{L}$.
Since $(\sigma(a_{i_2}^{-1}) - a_{i_2}^{-1}) \in L^{\times}$, we have $\bar{h} \in \mfa_{L}$.
\end{proof}

\begin{lemm}\label{gal_bij2}
The map $\prod_{l} \mfb_{l} \mapsto (\mfb_{l})_{l}$ is a bijection between
the set of ideals of $\LX$ which are $\sigma_{0}$-stable,
and the set of families $(\mfb_{l})_{l}$ where
$\mfb_{l}$ is an ideal of $\LlX$ and $\sigma_{0} \mfb_{l} \subset \mfb_{l+1}$ for all $l$.
\end{lemm}

\begin{proof}
This is clear.
\end{proof}

\begin{lemm}\label{gal_bij3}
For each $l$, we give an ideal $\mfb_{l} \subset \LlX$
such that $\sigma_{0} \mfb_{l} \subset \mfb_{l+1}$.
Then the restriction $\mfb_{l} \cap \FX$ is independent of $l$.
The same is also true if we replace $L_{l}$ by $\Sigma_{l}$.
\end{lemm}

\begin{proof}
We only prove the case of $L_{l}$.
Let $\iota$ be the natural injection $\FX \hookrightarrow \LX$
and $\pi_{l}$ the natural projection $\LX \twoheadrightarrow \LlX$ for each $l$.
For each $l$, $\sigma_{0}$ induces a morphism
$\LlX \rar L_{l+1}\X$; we also denote this by $\sigma_{0}$.
Then we have an equality $\pi_{l+1} \iota = \sigma_{0} \pi_{l} \iota$.
For any $h \in \mfb_{l} \cap \FX = (\pi_{l} \iota)^{-1} \mfb_{l}$,
we have $\pi_{l+1}(\iota(h)) = \sigma_{0}(\pi_{l}(\iota(h))) \in \sigma_{0} \mfb_{l} \subset \mfb_{l+1}$.
Hence $h \in (\pi_{l+1}\iota)^{-1}\mfb_{l+1} = \mfb_{l+1} \cap \FX$.
Therefore we obtain $\mfb_{l} \cap \FX \subset \mfb_{l+1} \cap \FX$.
Since the index set $\mathbb{Z}/d$ is a finite cyclic group, this inclusion is an equality.
\end{proof}

For any ring $R$, we denote by $\GL_{r/R}$ the $R$-group scheme of $r \times r$ invertible matrices.

\begin{prop}\label{gal_tri1}
$(1)$ Let $\phi : Z_{L} \rar \GL_{r/L}$ be the morphism of affine $L$-schemes defined by $u \mapsto \Psi^{-1} u$
for any $L$-algebra $S$ and any $S$-valued point $u \in Z(S)$.
Then $\phi$ factors through an isomorphism $\phi' : Z_{L} \rar \Gamma_{L}$ of affine $L$-schemes.

$(2)$ For any $l$ and $m$, let $\phi_{lm} : Z_{m,L_l} \rar \GL_{r/L_l}$ be
the morphism of affine $L_l$-schemes defined by $u \mapsto \Psi_{l}^{-1} u$
for any $L_l$-algebra $S$ and any $S$-valued point $u \in Z_{m}(S)$.
Then $\phi_{lm}$ factors through an isomorphism $\phi_{lm}' : Z_{m,L_l} \rar \Gamma_{m-l,L_l}$ of affine $L_l$-schemes.
\[\xymatrix{
Z_{L} \ar[rr]^{\phi ; u \mapsto \Psi^{-1} u} \ar@{.>}[dr]_{\phi'} & & \GL_{r/L} &
Z_{m,L_l} \ar[rr]^{\phi_{lm} ; u \mapsto \Psi_{l}^{-1} u} \ar@{.>}[dr]_{\phi_{lm}'} & & \GL_{r/L_l} \\
& \Gamma_{L} \ar@{^(->}[ru]_{natural} & &
& \Gamma_{m-l,L_l} \ar@{^(->}[ru]_{natural} & \\
}\]
\end{prop}

\begin{proof}
We prove only $(2)$.
Then $(1)$ can be proved by the same argument.
We define two $L_{l}$-algebra homomorphisms:
\begin{align}
\alpha_{l} &: \LlX \rar \LlX ; \ X \mapsto \Psi_{l}^{-1} X, \label{eq_alpha} \\
\bar{\alpha}_{lm} &: \LlX \xrightarrow{\alpha_{l}} \LlX
\twoheadrightarrow \LlX/\mfp_{m,L_l} = L_{l} \otimes_{E} \EX/\mfp_{m}. \label{eq_alphabar}
\end{align}
Then $\phi_{lm}$ corresponds to $\bar{\alpha}_{lm}$ on the level of coordinate rings.
Thus it is enough to show that $\alpha_{l}^{-1} \mfp_{m,L_l} = \mfq_{m-l,L_l}$.

For any $h(X) \in \LlX$, we have
$$\sigma_{1}\alpha_{l}h = \sigma_{1}(h(\Psi_{l}^{-1} X)) = h^{\sigma}((\sigma\Psi_{l}^{-1}) \Phi X)
= h^{\sigma}(\Psi_{l+1}^{-1} X) = \alpha_{l+1}h^{\sigma}(X) = \alpha_{l+1}\sigma_{0}h.$$
Therefore we have
$$\alpha_{l+1}^{-1} \sigma_{1} = \sigma_{0} \alpha_{l}^{-1}.$$
Since $\sigma_{1} \mfp_{m} \subset \mfp_{m+1}$ by Lemma \ref{gal_sigma}, we have an inclusion
$\sigma_{0} \alpha_{l}^{-1} \mfp_{m,L_l}
= \alpha_{l+1}^{-1} \sigma_{1} \mfp_{m,L_l} \subset \alpha_{l+1}^{-1} \mfp_{m+1,L_{l+1}}$.
Replacing $m$ by $m+l$, we obtain
$\sigma_{0} \alpha_{l}^{-1} \mfp_{m+l,L_l} \subset \alpha_{l+1}^{-1} \mfp_{m+l+1,L_{l+1}}$.
We consider the family of ideals $(\alpha_{l}^{-1} \mfp_{m+l,L_l})_{l}$.
Then for each $l$ and $m$, we have
$(\alpha_{l}^{-1} \mfp_{m+l,L_l} \cap \FX)_{L_l} = \alpha_{l}^{-1} \mfp_{m+l,L_l}$
by Lemmas \ref{gal_bij1}, \ref{gal_bij2} and \ref{gal_bij3}.
Again, replacing $m$ by $m-l$, we obtain an equality
$(\alpha_{l}^{-1} \mfp_{m,L_l} \cap \FX)_{L_l} = \alpha_{l}^{-1} \mfp_{m,L_l}$.

We consider $L_{l} \otimes_{E} L_{m}$ as an $L_{l}$-algebra via $f \mapsto f \otimes 1$,
and define an $L_{l}$-algebra homomorphism
$\widetilde{\mu} : \LlX \rar L_{l} \otimes_{E} L_{m} ; X_{ij} \mapsto 1 \otimes \Psi_{ijm}$.
Then $\mu_{lm} = \widetilde{\mu} \circ \alpha_{l}|_{\FX}$ and the map
$$\LlX \twoheadrightarrow \LlX/\mfp_{m,L_l}
= L_{l} \otimes_{E} \EX/\mfp_{m} \xhookrightarrow{\id \otimes \mu_{m}} L_{l} \otimes_{E} L_{m}$$
coincides with $\widetilde{\mu}$.
Therefore,
$$\mfq_{m-l} = \mfq_{l,m} = \ker \mu_{l,m} = \alpha_{l}|_{\FX}^{-1}(\ker \tilde{\mu})
= \alpha_{l}^{-1} \mfp_{m,L_l} \cap \FX.$$
Thus we have $\mfq_{m-l,L_l} = (\alpha_{l}^{-1} \mfp_{m,L_l} \cap \FX)_{L_l}
= \alpha_{l}^{-1} \mfp_{m,L_l}$.
\end{proof}

\begin{lemm}\label{gal_maximal}
$(1)$ The ideal $\mfp \subsetneq \EX$ is maximal among the proper $\sigma_{1}$-invariant ideals.

$(2)$ The family of ideals $(\mfp_{l})_{l}$ is maximal among the families of proper ideals $(\mfm_{l})_{l}$ of $\EX$
which satisfies $\sigma_{1} \mfm_{l} \subset \mfm_{l+1}$ for all $l$.
\end{lemm}

\begin{proof}
We prove only $(2)$.
Then $(1)$ can be proved by the same argument.
Let $(\mfm_{l})_{l}$ be a family of proper ideals of $\EX$ such that
$\mfp_{l} \subset \mfm_{l}$ and $\sigma_{1} \mfm_{l} \subset \mfm_{l+1}$ for all $l$.
Let $\alpha_{l}$ be the homomorphism (\ref{eq_alpha}).
We consider the family of ideals $(\alpha_{l}^{-1} \mfm_{l,L_l})_{l}$.
Since $\sigma_{0} \alpha_{l}^{-1} \mfm_{l,L_l}
= \alpha_{l+1}^{-1} \sigma_{1} \mfm_{l,L_l} \subset \alpha_{l+1}^{-1} \mfm_{l+1,L_{l+1}}$,
we can apply Lemma \ref{gal_bij2} to $(\alpha_{l}^{-1} \mfm_{l,L_l})_{l}$.
Then $\alpha_{l}^{-1} \mfm_{l,L_l} \cap \FX$ is independent of $l$,
and we take a maximal ideal $\mfa \subset \FX$ which contains this ideal.
We also have $(\alpha_{l}^{-1} \mfm_{l,L_l} \cap \FX)_{L_l} = \alpha_{l}^{-1} \mfm_{l,L_l}$
by Lemmas \ref{gal_bij1}, \ref{gal_bij2} and \ref{gal_bij3}.
Thus we obtain an inclusion $\alpha_{l}^{-1} \mfm_{l,L_l} \subset \mfa_{L_l}$.
We put $M := \FX / \mfa$ and define a morphism
$$\pi_{l} : \EX \hookrightarrow L_{l}\X \xrightarrow[\alpha_{l}^{-1}]{} L_{l}\X
\xrightarrow[\rho_{l}]{} L_{l}\X / \mfa_{L_l} \xrightarrow[\beta_{l}]{} L_{l} \otimes_{F} M,$$
where $\rho_{l}$ is the natural projection and
$\beta_{l} : L_{l}\X / \mfa_{L_l} \cong L_{l} \otimes_{F} \FX / \mfa = L_{l} \otimes_{F} M$.
Then we have $\mfm_{l} \subset \ker \pi_{l}$.

We define a $\sigma$-action on $L \otimes_{F} M$ by $\sigma \otimes \id$.
In $\GL_{r}(L_{l+1} \otimes_{F} M)$, we have
\begin{align*}
\sigma(\pi_{l}(X)) &= \sigma(\beta_{l}(\rho_{l}(\alpha_{l}^{-1}(X))))
= \beta_{l+1}(\rho_{l+1}(\sigma_{0}(\alpha_{l}^{-1}(X)))) = \beta_{l+1}(\rho_{l+1}(\alpha_{l+1}^{-1}(\sigma_{1}(X)))) \\
&= \pi_{l+1}(\sigma_{1}(X)) = \pi_{l+1}(\Phi X) = \Phi \pi_{l+1}(X).
\end{align*}
Set $\pi(X) := (\pi_{l}(X))_{l} \in \GL_{r}(L \otimes_{F} M)$.
Then we have $\sigma(\pi(X)) = \Phi \pi(X)$.
Since $(L \otimes_{F} M)^{\sigma} = M$,
we obtain $\delta := \pi(X)^{-1} \Psi \in \GL_{r}((L \otimes_{F} M)^{\sigma}) = \GL_{r}(M)$.
We define a $\delta$-action on $(E \otimes_{F} M)\X$ by $\delta \cdot h(X) := h(X \delta)$.
We extend $\pi_{l}$ to
$$\pi_{l}' : (E \otimes_{F} M)\X = \EX \otimes_{F} M \xrightarrow[\pi_{l} \otimes \id]{} L_{l} \otimes_{F} M.$$
Then we have $\mfp_{l} \otimes_{F} M \subset \mfm_{l} \otimes_{F} M \subset \ker \pi_{l}'
= \delta \cdot \ker(\nu_{l} \otimes \id_{M}) = \delta \cdot (\mfp_{l} \otimes_{F} M)$,
where the first equality is proved as follows:
For any $h(X) \in (E \otimes_{F} M)\X$,
$(\nu_{l} \otimes \id_{M})(h(X \delta^{-1})) = h(\Psi_{l} \Psi^{-1} \pi(X)) = h(\pi_{l}(X)) = \pi_{l}'(h(X))$.
Thus $h \in \ker \pi_{l}'$ is equivalent to $\delta^{-1} \cdot h \in \ker(\nu_{l} \otimes \id_{M})$.
Since $(E \otimes_{F} M)\X$ is a noetherian ring,
$\mfp_{l} \otimes_{F} M \subset \delta \cdot (\mfp_{l} \otimes_{F} M)$ implies
$\mfp_{l} \otimes_{F} M = \delta \cdot (\mfp_{l} \otimes_{F} M)$.
Therefore we have $\mfp_{l} \otimes_{F} M = \mfm_{l} \otimes_{F} M$.
Since $(E \otimes_{F} M)\X$ is faithfully flat over $\EX$, we have $\mfp_{l} = \mfm_{l}$.
\end{proof}

\begin{lemm}\label{gal_SFS}
$(1)$ Let $\mfb \subset \SX$ be an ideal which is $\sigma_{0}$-invariant.
Then we have $\mfb = (\mfb \cap \FX)_{\Sigma}$.

$(2)$ Let $\mfb_{l} \subset \SlX$ be ideals which satisfy $\sigma_{0} \mfb_{l} \subset \mfb_{l+1}$ for all $l$.
Then we have $\mfb_{l} = (\mfb_{l} \cap \FX)_{\Sigma_l}$.
\end{lemm}

\begin{proof}
We prove only $(2)$.
Then $(1)$ can be proved by the same argument.
Set $\mfa := \mfb_{l} \cap \FX$, which is independent of $l$ by Lemma \ref{gal_bij3}.
Suppose that $\mfa_{\Slz} \subsetneq \mfb_{l_0}$ for some $l_{0}$.
Let $(g_{i})_{i \in I}$ be an $F$-basis of $\FX$ such that
$I = I_{1} \amalg I_{\mfa}$ and $\mfa = \oplus_{i \in I_{\mfa}} F g_{i}$.
Then we have
$$\SlX = (\oplus_{i \in I_{1}} \Sigma_{l} g_{i}) \oplus (\oplus_{i \in I_{\mfa}} \Sigma_{l} g_{i})
= (\oplus_{i \in I_{1}} \Sigma_{l} g_{i}) \oplus \mfa_{\Sl}.$$
Since $\mfa_{\Slz} \subsetneq \mfb_{l_0}$, we can take a minimal finite set $J \subset I_{1}$
so that $\mfb_{l_0} \cap (\oplus_{i \in J} \Slz g_{i}) \neq 0$.
By the injectivity of $\sigma_{0}$ and the inclusion
$\sigma_{0}(\mfb_{l} \cap \oplus_{i \in J} \Sl g_{i}) \subset \mfb_{l+1} \cap \oplus_{i \in J} \Sigma_{l+1} g_{i}$,
$J$ has the same properties for all $l$.
We fix $j \in J$ and consider the ideal of $\Sl \cong \EX / \mfp_{l}$:
$$\mfm_{l} := \{b \in \Sl | \textrm{ there exists }
\sum_{i \in J} b_{i} g_{i} \in \mfb_{l} \cap (\oplus_{i \in J} \Sl g_{i}) \textrm{ such that } b_{j} = b \}.$$
Then $\mfm_{l}$ is a non-zero ideal by the minimality of $J$,
and it is clear that $\sigma \mfm_{l} \subset \mfm_{l+1}$.
By Lemma \ref{gal_sigma} we have $\sigma \nu_{l} = \nu_{l+1} \sigma_{1}|_{\EX}$.
Hence we can apply Lemma \ref{gal_maximal} to the inverse image of $(\mfm_{l})_{l}$ in $\EX$.
Therefore we have $\mfm_{l} = \Sl$.
Thus for each $l$, there exists an element
$h_{l} = \sum_{i \in J} b_{li} g_{i} \in \mfb_{l} \cap (\oplus_{i \in J} \Sl g_{i})$ such that $b_{lj} = 1$.
Then we have
$\mfb_{l+1} \ni \sigma_{0} h_{l} - h_{l+1} = \sum_{i \in J \setminus \{j\}} (\sigma b_{li} - b_{l+1,i}) g_{i}$.
By the minimality of $J$, $\sigma b_{li} = b_{l+1,i}$ for all $i$ and $l$.
We put $b_{i} := (b_{li})_{l} \in \prod_{l} \Sl$.
Then we have $\sigma b_{i} = (\sigma b_{l-1,i})_{l} = (b_{li})_{l} = b_{i}$.
Hence $b_{i} \in F$ and $b_{li} \in F$ via the $l$-th projection.
Then $0 \neq h_{l} = \sum_{i \in J} b_{li} g_{i} \in \mfb_{l} \cap \FX = \mfa = \oplus_{i \in I_{\mfa}} F g_{i}$.
This contradicts $J \cap I_{\mfa} = \emptyset$.
\end{proof}

\begin{prop}\label{gal_tri2}
$(1)$ Let $\psi : Z \times_{E} Z \rar Z \times_{E} \GL_{r/E}$ be
the morphism of affine $E$-schemes defined by $(u,v) \mapsto (u,u^{-1}v)$
for any $E$-algebra $S$ and any $S$-valued point $(u,v) \in Z(S) \times Z(S)$.
Then $\psi$ factors through an isomorphism
$\psi' : Z \times_{E} Z \rar Z \times_{E} \Gamma_{E}$ of affine $E$-schemes.

$(2)$ For any $l$ and $m$,
let $\psi_{lm} : Z_{l} \times_{E} Z_{l+m} \rar Z_{l} \times_{E} \GL_{r/E}$ be
the morphism of affine $E$-schemes defined by $(u,v) \mapsto (u,u^{-1}v)$
for any $E$-algebra $S$ and any $S$-valued point $(u,v) \in Z_{l}(S) \times Z_{l+m}(S)$.
Then $\psi$ factors through an isomorphism
$\psi_{lm}' : Z_{l} \times_{E} Z_{l+m} \rar Z_{l} \times_{E} \Gamma_{m,E}$ of affine $E$-schemes.
\[\xymatrix{
Z \! \times \! Z \ar[rr]^{\psi ; (u,v) \mapsto (u,u^{-1}v)} \ar@{.>}[dr]_{\psi'} & &
Z \! \times \! \GL_{r/E} & Z_{l} \! \times \! Z_{l+m}
\ar[rr]^{\psi_{lm} ; (u,v) \mapsto (u,u^{-1}v)} \ar@{.>}[dr]_{\psi_{lm}'} & & Z_{l} \! \times \! \GL_{r/E} \\
& Z \! \times \! \Gamma_{E} \ar@{^(->}[ru]_{natural} & &
& Z_{l} \! \times \! \Gamma_{m,E} \ar@{^(->}[ru]_{natural} & \\
}\]
\end{prop}

\begin{proof}
We prove only $(2)$.
Then $(1)$ can be proved by the same argument.
Let $\alpha_{l}$ and $\bar{\alpha}_{lm}$ be the homomorphisms (\ref{eq_alpha}) and (\ref{eq_alphabar}).
We restrict the domain of $\bar{\alpha}_{lm}$ to $\SlX \cong \EX / \mfp_{l} \otimes_{E} \EX$
and the target of $\bar{\alpha}_{lm}$ to $\Sl \otimes_{E} \EX / \mfp_{m} \cong \EX / \mfp_{l} \otimes_{E} \EX / \mfp_{m}$. 
Then $\psi_{lm}$ corresponds to $\bar{\alpha}_{l,l+m}$ on the level of coordinate rings.
Hence it is enough to show that $\ali \mfp_{m+l,\Sl} = \mfq_{m,\Sl}$.

Since we have an inclusion
$\sigma_{0} \ali \mfp_{m+l,\Sl} = \alpha_{l+1}^{-1} \sigma_{1} \mfp_{m+l,\Sl} \subset \alpha_{l+1}^{-1} \mfp_{m+l+1,\Sigma_{l+1}}$,
the ideal $\mfa_{m} := \ali \mfp_{m+l,\Sl} \cap \FX$ is independent of $l$ by Lemma \ref{gal_bij3},
and we can apply Lemma \ref{gal_SFS} to $(\ali \mfp_{m+l,\Sl})_{l}$.
Then we have $\ali \mfp_{m+l,\Sl} = (\ali \mfp_{m+l,\Sl} \cap \FX)_{\Sl} = \mfa_{m,\Sl}$.
On the other hand we have $\ali \mfp_{m+l,L_l} = \mfq_{m,L_l}$ by Proposition \ref{gal_tri1}.
Therefore $\mfq_{m,L_l} = \ali \mfp_{m+l,L_l} = \ali((\mfp_{m+l,\Sl})_{L_l})
= (\ali \mfp_{m+l,\Sl})_{L_l} = (\mfa_{m,\Sl})_{L_l} = \mfa_{m,L_l}$.
Then we have $\mfq_{m} = \mfa_{m}$ since $\LlX$ is faithfully flat over $\FX$.
Thus we obtain $\ali \mfp_{m+l,\Sl} = \mfa_{m,\Sl} = \mfq_{m,\Sl}$.
\end{proof}

The next lemma is proved by an elementary argument.
Thus we omit the proof.
This lemma is applied to the $S$-valued points of the diagrams in Proposition \ref{gal_tri2}
where $S$ is an $\bar{E}$-algebra.
\begin{lemm}\label{gal_gp}
$(1)$ Let $G$ be a group, $A$ and $B$ be non-empty subsets of $G$ such that the map
$$\psi : A \times A \rar A \times G ; \ (u,v) \mapsto (u,u^{-1}v)$$
factors through a bijection $\psi' : A \times A \rar A \times B$.
Then $B$ is a subgroup of $G$,
$A$ is stable under right-multiplication by elements of $B$
and $A$ becomes a $B$-torsor.

$(2)$ Let $G$ be a group, $A_{l}$ and $B_{m}$ be non-empty subsets of $G$ such that the map
$$\psi_{lm} : A_{l} \times A_{l+m} \rar A_{l} \times G ; \ (u,v) \mapsto (u,u^{-1}v)$$
factors through a bijection $\psi_{lm}' : A_{l} \times A_{l+m} \rar A_{l} \times B_{m}$ for each $l$ and $m$.
Then $B_{0}$ is a subgroup of $G$,
$A_{l}$ is stable under right-multiplication by elements of $B_{0}$
and $A_{l}$ becomes a $B_{0}$-torsor for each $l$.
Moreover, for any $u \in A_{l}$ and $v \in A_{l+m}$ there exists an element $y \in B_{m}$ such that $v = u y$.
The multiplication in $G$ induces maps $A_{l} \times B_{m} \rar A_{l+m}$ and $B_{m} \times B_{m'} \rar B_{m+m'}$,
and the inversion in $G$ induces a map $B_{m} \rar B_{-m}$.
\end{lemm}

By Proposition \ref{gal_tri2} and Lemma \ref{gal_gp}, we have surjective maps
\begin{align}
Z_{l}(S) \times Z_{l+m}(S) \rar \Gamma_{m}(S) ; \ (u,v) \mapsto u^{-1} v, \label{eq_gp1} \\
Z_{l}(S) \times \Gamma_{m}(S) \rar Z_{l+m}(S) ; \ (x,y) \mapsto x y \label{eq_gp2}
\end{align}
for any $\bar{E}$-algebra $S$.

\begin{theo}\label{gal_torsor}
$(1)$ The $F$-scheme $\Gamma$ is a closed $F$-subgroup scheme of $\GL_{r/F}$,
the $E$-scheme $Z$ is stable under right multiplication by $\Gamma_{E}$ and is a $\Gamma_{E}$-torsor.

$(2)$ The $F$-scheme $\Gamma_{0}$ is a closed $F$-subgroup schemes of $\GL_{r/F}$,
the $E$-scheme $Z_{l}$ is stable under right multiplication by $\Gamma_{0,E}$
and is a $\Gamma_{0,E}$-torsor for each $l$.

$(3)$ The $F$-scheme $\Gamma_{m}$ is stable under right and left multiplications by $\Gamma_{0}$
and is a $\Gamma_{0}$-torsor for each $m$.
\end{theo}

\begin{proof}
We prove only $(2)$.
Then $(1)$ and $(3)$ can be proved by the same argument.
By Proposition \ref{gal_tri2}, we have a bijection
$Z_{l}(S) \times Z_{l+m}(S) \rar Z_{l}(S) \times \Gamma_{m}(S) ; (u,v) \mapsto (u,u^{-1}v)$
for any $\bar{E}$-algebra $S$.
Since $Z_{l}(S)$ is non-empty,
Lemma \ref{gal_gp} implies that $\Gamma_{0,\bar{E}}$ is a closed subgroup scheme of $\GL_{r/\bar{E}}$
and $Z_{l,\bar{E}}$ is a $\Gamma_{0,\bar{E}}$-torsor.
Therefore $\Gamma_{0}$ is a closed subgroup scheme of $\GL_{r/F}$
by the faithfully flatness of the inclusion $F \rar \bar{E}$.
Similarly, $Z_{l}$ is a $\Gamma_{0,E}$-torsor by the faithfully flatness of the inclusion $E \rar \bar{E}$.
\end{proof}

\begin{theo}\label{gal_smooth}
$(a)$ The $E$-schemes $Z$ and $Z_{l}$ are smooth.

$(b)$ The $F$-schemes $\Gamma$ and $\Gamma_{m}$ are smooth.

$(c)$ If $E$ is algebraically closed in the fraction field $\Lambda_{l_0}$ of $\Slz$ for some $l_{0}$,
then $Z_{l}$ and $\Gamma_{m}$ are absolutely irreducible.

$(d)$ $\dim \Gamma = \dim \Gamma_{m} = \tdeg_{E}\Lambda_{l}$.
\end{theo}

\begin{proof}
$(a)$, $(b)$ Since $L_{l} / E$ is a separable extension,
$\Lambda / E$ is also a separable extension,
where $\Lambda = \Frac(\Sigma)$ the total ring of fractions of $\Sigma$.
Thus for any field extension $\Omega / E$, $\Lambda \otimes_{E} \Omega$ is reduced.
Therefore $\Sigma \otimes_{E} \Omega$ is reduced and $Z = \Spec \Sigma$ is absolutely reduced.
Since $\Gamma_{\bar{E}} \cong Z_{\bar{E}}$, $\Gamma$ is absolutely reduced.
Since $\Gamma$ is an algebraic group,
the property that $\Gamma$ is absolutely reduced implies that $\Gamma$ is smooth.
Again since $\Gamma_{\bar{E}} \cong Z_{\bar{E}}$, we have that $Z$ is smooth.
The statements of $Z_{l}$ and $\Gamma_{m}$ are proved similarly.

$(c)$ For any field extension $\Omega / E$,
$\Lambda_{l_0} \otimes_{E} \Omega$ is an integral domain by the assumption.
Therefore $\Sigma_{l_0} \otimes_{E} \Omega$ is an integral domain and $Z_{l_0}$ is absolutely integral.
Since $Z_{l,\bar{E}} \cong \Gamma_{m,\bar{E}}$ for all $l$ and $m$,
$Z_{l}$ and $\Gamma_{m}$ are all absolutely integral.

$(d)$ We have an equality
$\dim \Gamma = \dim \Gamma_{m} = \dim \Gamma_{0} = \dim Z_{l} = \tdeg_{E}\Lambda_{l}$.
\end{proof}

\begin{coro}
$(1)$ There exists a divisor $d'$ of $d$ such that if $l \equiv l' \pmod {d'}$ then $Z_{l} = Z_{l'}$
and if $l \not\equiv l' \pmod {d'}$ then $Z_{l} \cap Z_{l'} = \emptyset$.

$(2)$ If $m \equiv m' \pmod {d'}$ then $\Gamma_{m} = \Gamma_{m'}$
and if $m \not\equiv m' \pmod {d'}$ then $\Gamma_{m} \cap \Gamma_{m'} = \emptyset$.
\end{coro}

Therefore we can write $Z = \amalg_{l \in \mathbb{Z}/d'} Z_{l}$,
$\Sigma = \prod_{l \in \mathbb{Z}/d'} \Sl$, $\Lambda = \prod_{l \in \mathbb{Z}/d'} \Lambda_{l}$
and $\Gamma = \amalg_{m \in \mathbb{Z}/d'} \Gamma_{m}$.

\begin{proof}
$(1)$ Since $Z_l$ is a $\Gamma_{0,E}$-torsor and absolutely reduced for all $l$,
it is clear that $Z_{l} = Z_{l'}$ or $Z_{l} \cap Z_{l'} = \emptyset$.
We have the surjective map (\ref{eq_gp2}) : $Z_{l}(\bar{E}) \times \Gamma_{1}(\bar{E}) \rar Z_{l+1}(\bar{E})$.
Therefore if $Z_{l} =Z_{l'}$, then $Z_{l+1}(\bar{E}) = Z_{l'+1}(\bar{E})$.
Hence if we take $d'$ to be the minimum positive integer such that $Z_{0} = Z_{d'}$,
then $d'$ satisfies the desired properties.

$(2)$ By the same argument of the proof of $(1)$,
there exists a divisor $d''$ of $d$ which is the period of $(\Gamma_{m})_{m}$.
Then by the map (\ref{eq_gp1}), we have
$$\Gamma_{m}(\bar{E}) = Z_{l}(\bar{E})^{-1} Z_{l+m}(\bar{E}) = Z_{l}(\bar{E})^{-1} Z_{l+m+d'}(\bar{E}) = \Gamma_{m+d'}(\bar{E}).$$
This means that $d'' | d'$.
By the map (\ref{eq_gp2}), we have
$$Z_{l}(\bar{E}) = Z_{l}(\bar{E}) \Gamma_{0}(\bar{E}) = Z_{l}(\bar{E}) \Gamma_{d''}(\bar{E}) = Z_{l+d''}(\bar{E}).$$
This means that $d' | d''$.
\end{proof}

\subsection{$\Gamma$-action}\label{gal_action}

For any $F$-algebras $R$ and $S$, we set $S\RR := R \otimes_{F} S$.
In particular, if $R = F'$ is a field, we set $S' := S^{(F')}$.
If $\sigma$ acts on $S$, we define the $\sigma$-action on $S\RR$ by $\id \otimes \sigma$.
Note that, if $S^{\sigma} = F$, then we have $(S\RR)^{\sigma} = R$.
Let $\aus(\Sigma\RR / E\RR)$ denote the group of automorphisms of $\Sigma\RR$ over $E\RR$ that commute with $\sigma$.
Similarly we define $\aus(\Lambda\RR/ E\RR)$.
For any $\gamma \in \Gamma(R)$, we obtain an automorphism $Z_{E\RR} \rar Z_{E\RR} ; x \mapsto x\gamma$.
On the level of coordinate rings, this corresponds to an automorphism
$\Sigma\RR \rar \Sigma\RR ; h(\Psi) \mapsto \gamma . h(\Psi) := h(\Psi \gamma)$.
Note that $\Sigma\RR = E\RR \otimes_{E} \Sigma = E\RR [\Psi,\Delta(\Psi)^{-1}] \cong E\RR\X / \mfp_{E\RR}$.
Thus we have a group homomorphism $\kappa_{R} : \Gamma(R) \rar \Aut(\Sigma\RR / E\RR)$.

\begin{lemm}\label{gal_g=a}
$(1)$ For any $F$-algebra $R$, the map $\kappa_{R}$ induces an isomorphism
$\Gamma(R) \xrightarrow{\sim} \aus(\Sigma\RR / E\RR)$.
Its inverse is the map $\alpha \mapsto \Psi^{-1} (\alpha \Psi_{ij})_{ij}$.

$(2)$ $\aus(\Sigma / E) \cong \aus(\Lambda / E)$.

$(3)$ If $\Lambda_{l} / F$ is a regular extension
$($i.e. separable extension and $F$ is algebraically closed in $\Lambda_{l})$
for all $l$ and $F' / F$ is an algebraic extension of fields,
then we have $\aus(\Sigma' / E') \cong \aus(\Lambda' / E')$.
\end{lemm}

\begin{proof}
$(1)$ For any $\gamma \in \Gamma(R)$ and $h(\Psi) \in \Sigma\RR$, we have
$\sigma(\gamma . h(\Psi)) = \sigma(h(\Psi \gamma)) = h^{\sigma}(\sigma(\Psi \gamma)) = h^{\sigma}(\Phi \Psi \gamma)
= \gamma . (h^{\sigma}(\Phi \Psi)) = \gamma . (h^{\sigma}(\sigma \Psi)) = \gamma . (\sigma(h(\Psi)))$.
Hence $\kappa_{R}(\gamma)$ commutes with $\sigma$.
Suppose that $\kappa_{R}(\gamma)$ is the identity.
Then $h(\Psi \gamma) = h(\Psi)$ for any $h(\Psi) \in \Sigma\RR$.
In particular if we take $h(\Psi) = \Psi_{ij}$ for each $i$ and $j$,
then we obtain $\Psi \gamma = \Psi$ in $\GL_{r}(\Sigma\RR)$.
Therefore $\gamma = 1$ and this means that $\kappa_{R}$ is injective.
Conversely, let $\alpha \in \aus(\Sigma\RR / E\RR)$ be any element.
Then $\alpha$ corresponds to an automorphism $\bar{\alpha} : Z_{E\RR} \rar Z_{E\RR}$,
and $\bar{\alpha}$ maps the $\Sigma\RR$-valued point $\Psi$ to $(\alpha \Psi_{ij})_{ij}$.
By Theorem \ref{gal_torsor}, there exists an element $\gamma \in \Gamma(\Sigma\RR)$
such that $\Psi \gamma = (\alpha \Psi_{ij})_{ij}$.
Then for any $h(\Psi) \in \Sigma\RR$,
we have $\alpha(h(\Psi)) = h((\alpha \Psi_{ij})_{ij}) = h(\Psi \gamma)$.
Thus we obtain
$\sigma(\gamma . h(\Psi)) = \sigma(\alpha(h(\Psi))) = \alpha(\sigma(h(\Psi)))
= \alpha(h^{\sigma}(\Phi \Psi)) = \gamma . h^{\sigma}(\Phi \Psi) = h^{\sigma}(\Phi \Psi \gamma)$.
On the other hand,
$\sigma(\gamma . h(\Psi)) = \sigma(h(\Psi \gamma)) = h^{\sigma}((\sigma \Psi)(\sigma \gamma)) = h^{\sigma}(\Phi \Psi (\sigma \gamma))$.
If we take $h(\Psi) = \Psi_{ij}$ for each $i$ and $j$,
we obtain $\Phi \Psi (\sigma \gamma) = \Phi \Psi \gamma$.
Hence $\sigma \gamma = \gamma$.
Therefore we have $\gamma \in \Gamma(R)$ and $\kappa_{R}(\gamma) = \alpha$.

$(2)$ Since $\Lambda = \Frac(\Sigma)$,
any automorphism of $\Sigma$ extends uniquely to an automorphism of $\Lambda$.
Conversely if $\alpha \in \aus(\Lambda / E)$,
then $\sigma(\alpha(\Psi)) = \alpha(\sigma \Psi) = \alpha(\Phi \Psi) = \Phi \cdot \alpha(\Psi)$ in $\GL_{r}(L)$.
Thus we have $\alpha(\Psi) \in \Psi \cdot \GL_{r}(F)$.
This implies that $\alpha(\Sigma) \subset \Sigma$.
Similarly, we have $\alpha^{-1}(\Sigma) \subset \Sigma$.
Therefore $\alpha(\Sigma) = \Sigma$.

$(3)$ Since $\Lambda_{l} / F$ is a regular extension, so is $E / F$.
Since $F' / F$ is an algebraic extension,
$E'$ and $\Lambda_{l}'$ are fields and $\Lambda_{l}' = \Frac(\Sl')$.
Therefore $\Lambda' = \prod_{l \in \mathbb{Z}/d'} \Lambda_{l}'$ is a finite product of fields and $\Lambda' = \Frac(\Sigma')$.
Then, the proof is the same as $(2)$.
\end{proof}

We prepare some lemmas about Zariski density.
\begin{lemm}\label{gal_closure}
Let $\Omega / k$ be a field extension such that $\Omega$ is an algebraically closed field,
$X$ an algebraic variety over $k$ and $Y$ a closed subvariety of $X_{\Omega}$.
If $X(k) \cap Y(\Omega)$ is Zariski dense in $Y$, then $Y$ is defined over $k$,
i.e. there exists some algebraic variety $Y_{0}$ over $k$ such that $Y = Y_{0,\Omega}$.
\end{lemm}

\begin{proof}
We may assume that $X$ is affine.
Let $k[X]$ be the coordinate ring of $X$ and $\Omega[X]$ the coordinate ring of $X_{\Omega}$.
Then we have $\Omega[X] = \Omega \otimes_{k} k[X]$.
Let $\mfa \subset \Omega[X]$ be the defining ideal of $Y$,
$\mfa_{k} := \mfa \cap k[X]$ and $\mfa_{k,\Omega} := \mfa_{k} \cdot \Omega[X]$.
We need to show that $\mfa_{k,\Omega} = \mfa$.
Thus we assume that $\mfa_{k,\Omega} \subsetneq \mfa$.
Let $(g_{i})_{i \in I'}$ be a $k$-basis of $k[X]$
such that $(g_{i})_{i \in I}$ is a $k$-basis of $\mfa_{k}$ for some $I \subset I'$.
We also take $(c_{j})_{j \in J}$ to be a $k$-basis of $\Omega$.
Since $\mfa_{k,\Omega} \subsetneq \mfa$,
there exists a non-zero element $f = \sum_{i \in I' \setminus I} a_{i} g_{i} \in \mfa$ where $a_{i} \in \Omega$.
Write $a_{i} = \sum_{j} \alpha_{ij}c_{j} \ (\alpha_{ij} \in k)$.
Then we can write $f = \sum_{j} c_{j} \sum_{i \in I' \setminus I} \alpha_{ij} g_{i}$.
For any $x \in X(k) \cap Y(\Omega)$,
we have $\sum_{j} c_{j} \sum_{i \in I' \setminus I} \alpha_{ij} g_{i}(x) = f(x) = 0$.
Since $(c_{j})_{j}$ is linearly independent over $k$,
we have $\sum_{i \in I' \setminus I} \alpha_{ij} g_{i}(x) = 0$ for all $j$.
By the density assumption,
we obtain $\sum_{i \in I' \setminus I} \alpha_{ij} g_{i}(x) = 0$ for all $x \in Y(\Omega)$.
Therefore $\sum_{i \in I' \setminus I} \alpha_{ij} g_{i} \in \mfa \cap k[X] = \mfa_{k} = \oplus_{i \in I} k g_{i}$ for all $j$.
Since $(g_{i})_{i \in I'}$ is linearly independent over $k$, we have $\alpha_{ij} = 0$.
Thus $f = 0$, which is a contradiction.
\end{proof}

\begin{coro}\label{gal_ext}
Let $\Omega / k$ be a field extension such that $\Omega$ is an algebraically closed field
and $X$ an algebraic variety over $k$.
If $X(k)$ is Zariski dense in $X$, then $X(k)$ is Zariski dense in $X_{\Omega}$.
\end{coro}

\begin{proof}
We take $Y$ to be the Zariski closure of $X(k)$ in $X_{\Omega}$.
Then there exists some algebraic variety $Y_{0}$ over $k$
such that $Y = Y_{0,\Omega}$ by Lemma \ref{gal_closure}.
It is clear that $X(k) \subset Y_{0}(k)$.
Hence we have $Y_{0} = X$ since $X(k)$ is Zariski dense in $X$.
Therefore we have $Y = X_{\Omega}$.
\end{proof}

\begin{lemm}\label{gal_prod}
Let $\Omega / k$ be a field extension, $X_{1}$ an algebraic variety over $\Omega$
and $X_{2}$ an algebraic varieties over $k$.
If $X_{2}(k)$ is Zariski dense in $X_{2}$,
then $X_{1}(\bar{\Omega}) \times X_{2}(k)$ is Zariski dense in $X_{1} \times_{\Omega} X_{2,\Omega}$,
where $\bar{\Omega}$ is an algebraic closure of $\Omega$.
\end{lemm}

\begin{proof}
Let $V$ be the Zariski closure of $X_{1}(\bar{\Omega}) \times X_{2}(k)$ in $X_{1} \times_{\Omega} X_{2,\Omega}$.
We assume that $V \subsetneq X_{1} \times_{\Omega} X_{2,\Omega}$.
Then we have an element
$(x,y) \in (X_{1}(\bar{\Omega}) \times X_{2}(\bar{\Omega})) \setminus V(\bar{\Omega})$.
Therefore we have $(\{x\} \times X_{2,\bar{\Omega}}) \cap V_{\bar{\Omega}} \subsetneq \{x\} \times X_{2,\bar{\Omega}}$ and
$\{x\} \times X_{2}(k) \subset (\{x\} \times X_{2}(\bar{\Omega})) \cap V(\bar{\Omega})$.
On the other hand, since $X_{2}(k)$ is Zariski dense in $X_{2}$,
$X_{2}(k)$ is Zariski dense in $X_{2,\bar{\Omega}}$ by Corollary \ref{gal_ext}.
Then $\{x\} \times X_{2}(k)$ is Zariski dense in $\{x\} \times_{\bar{\Omega}} X_{2,\bar{\Omega}}$.
This is a contradiction.
\end{proof}

\begin{theo}\label{gal_gammafix}
Let $F' / F$ be an algebraic extension of fields such that $\Gamma(F')$ is Zariski dense in $\Gamma_{F'}$.
Assume that $F' = F$ or $\Lal / F$ is a regular extension for all $l$.
Then we have $(\La')^{\Gamma(F')} = E'$ and $\La \cap (\La')^{\Gamma(F')} = E$.
\end{theo}

\begin{proof}
The second part follows from the first part and the assumptions.
Thus we prove the first part.
In the proof of this theorem, we regard $l$ as an element of the index set $\mathbb{Z}/d'$.
We take any element $f = (f_{l})_{l} \in (\La')^{\Gamma(F')} \subset \prod_{l \in \mathbb{Z}/d'} \Lal'$,
and consider $f_{l}$ as a rational function of $Z_{l,E'}$ to $\Af_{E'}$.
Then, for some non-empty open affine set $U_{l} \subset Z_{l,E'}$,
$f_{l}$ can be regarded as a morphism $f_{l} : U_{l} \rar \Af_{E'}$.
By Proposition \ref{gal_tri2}, we have an isomorphism
$Z_{E'} \times_{E'} \Gamma_{E'} \rar Z_{E'} \times_{E'} Z_{E'} ; (x,y) \mapsto (x,xy)$.
We set $U \subset Z_{E'} \times_{E'} \Gamma_{E'}$ to be the open subset corresponding to
$\amalg_{l} U_{l} \times_{E'} \amalg_{l} U_{l}$ via this isomorphism,
and consider the two maps
$$
\begin{CD}
g_{i} : U @>>> \amalg U_{l} \times_{E'} \amalg U_{l} @>\pi_{i}>> \amalg U_{l} @>f=(f_{l})>> \Af_{E'}
\end{CD}
$$
where $i=1,2$ and $\pi_{i}$ is the $i$-th projection.
Let $S$ be an algebraic closure of $E'$.
Then for any $(x,y) \in (Z(S) \times \Gamma(F')) \cap U(S)$,
we have $g_{1}(x,y) = f(\pi_{1}(x,xy)) = f(x)$ and $g_{2}(x,y) = f(\pi_{2}(x,xy)) = f(xy) = f(x)$
since $f$ is fixed by $\Gamma(F')$.
Since $\Gamma(F')$ is Zariski dense in $\Gamma_{F'}$,
$Z(S) \times \Gamma(F')$ is Zariski dense in $Z_{E'} \times_{E'} \Gamma_{E'}$ by Lemma \ref{gal_prod}.
Then $(Z(S) \times \Gamma(F')) \cap U(S)$ is Zariski dense in $U$.
Thus we have $g_{1} = g_{2}$, and this means $f \pi_{1} = f \pi_{2}$.
By considering on the level of coordinate rings, it is clear that $f \in E'$ since $E'$ is a field.
\end{proof}

\begin{coro}\label{gal_loc}
If $F$ is a local field, and each connected component of $\Gamma$ has an $F$-valued point,
then $\La^{\Gamma(F)} = E$.
\end{coro}

\begin{proof}
Take any connected component $\Gamma'$ of $\Gamma$.
Then there exists an $F$-valued point $x \in \Gamma'(F)$ by the assumption,
and $\Gamma'$ is smooth by Lemma \ref{gal_smooth}.
By the implicit function theorem, there exists an open neighborhood of $x$ in $\Gamma'(F)$
which is isomorphic to some open subset of $F^{\dim \Gamma}$.
Since $\Gamma'$ is irreducible, this implies that $\Gamma'(F)$ is Zariski dense in $\Gamma'$.
Hence we conclude that $\Gamma(F)$ is Zariski dense in $\Gamma$.
Then this corollary follows from Theorem \ref{gal_gammafix}.
\end{proof}

\section{The group $\Gamma$ and $\vp$-modules}\label{sec_p-d}

\subsection{General case}
In this subsection, we use the notations defined in Section \ref{sec_phi-t},
and fix a $\sigma$-admissible triple $(F,E,L)$.
Let $M \in \phm_{E}^{L}$ be an $L$-trivial $\varphi$-module over $E$ of rank $r$,
$\TM$ the Tannakian subcategory of $\phm_{E}^{L}$ generated by $M$,
$V_{M} : \TM \rar \vsp(F)$ the fiber functor of $\TM$
and $\Gamma_{M}$ the Tannakian Galois group of $(\TM, V_{M})$.
We fix $\mbm \in \Mat_{r \times 1}(M)$ an $E$-basis of $M$.
Then there exist matrices $\Phi \in \GL_{r}(E)$ and $\Psi \in \GL_{r}(L)$
such that $\varphi \mbm = \Phi \mbm$ and $\sigma \Psi = \Phi \Psi$.
We define $\Gamma, \Sigma, \dots$ as in Section \ref{sec_gal} for $\Phi$ and $\Psi$.
In this subsection, we show that there exists an equivalence of categories
$\TM \xrightarrow{\sim} \Rep(\Gamma,F)$ under some assumptions.
Note that $\Sigma$ and $\Sl$ are independent of the choice of $\mbm$ and $\Psi$ by Proposition \ref{phi_psi}.
If $N \in \TM$ and $s$ is the rank of $N$ over $E$, we use the notation $\mbn \in \Mat_{s \times 1}(N)$
and $\Psi_{N} \in \GL_{s}(L)$ for an $E$-basis of $N$ and a fundamental matrix respectively.
For any $F$-algebras $S$ and $R$, we set $S\RR := R \otimes_{F} S$.

\begin{prop}\label{p-d_psi-n}
For any $N \in \TM$, we have $\Psi_{N} \in \GL_{s}(\Sigma)$.
\end{prop}

\begin{proof}
Let $N$ and $N'$ be objects in $\TM$.
Set $s := \dim_{E} N$ and $s' := \dim_{E}$
and assume that $\Psi_{N} \in \GL_{s}(\Sigma)$ and $\Psi_{N'} \in \GL_{s'}(\Sigma)$.
Since we can take
$\Psi_{N \oplus N'} = \Psi_{N} \oplus \Psi_{N'}$, $\Psi_{N \otimes N'} = \Psi_{N} \otimes \Psi_{N'}$
and $\Psi_{N^{\vee}} = (\Psi_{N}^{-1})^{\mathrm{tr}}$,
we have that $\Psi_{N \oplus N'}$, $\Psi_{N \otimes N'}$
and $\Psi_{N^{\vee}}$ are invertible matrices with coefficients in $\Sigma$.
We have to show that if $0 \rar N' \rar N \rar N'' \rar 0$ is an exact sequence in $\TM$ and $\Psi_{N} \in \GL_{s}(\Sigma)$,
then $\Psi_{N'} \in \GL_{s'}(\Sigma)$ and $\Psi_{N''} \in \GL_{s''}(\Sigma)$.
Let $\mbn$, $\mbn'$ and $\mbn''$ be $E$-bases of $N$, $N'$ and $N''$ such that
$$\mbn = \begin{bmatrix} \mbn' \\ \tilde{\mbn}'' \\ \end{bmatrix}$$
where $\tilde{\mbn}''$ is a lift of $\mbn''$.
Since $V_{M}$ is exact, we have an exact sequence
$$0 \rar V_{M}(N') \rar V_{M}(N) \rar V_{M}(N'') \rar 0.$$
Let $\mbx$, $\mbx'$ and $\mbx''$ be $F$-bases of $V_{M}(N)$, $V_{M}(N')$ and $V_{M}(N'')$ such that
$$\mbx = \begin{bmatrix} \mbx' \\ \tilde{\mbx}'' \\ \end{bmatrix}$$
where $\tilde{\mbx}''$ is a lift of $\mbx''$.
By Proposition \ref{phi_basis}, there exist matrices $A \in \GL_{s}(F)$,
$A' \in \GL_{s'}(F)$ and $A'' \in \GL_{s''}(F)$ such that
$\Psi_{N}^{-1} \mbn = A \mbx$, $\Psi_{N'}^{-1} \mbn' = A' \mbx'$ and $\Psi_{N''}^{-1} \mbn'' = A'' \mbx''$.
Consider the exact sequence
$$0 \rar L \otimes_{E} N' \rar L \otimes_{E} N \rar L \otimes_{E} N'' \rar 0.$$
Since both $\Psi_{N''} A'' \tilde{\mbx}''$ and $\tilde{\mbn}''$ are mapped to $\mbn''$
and $\mbx'$ is an $L$-basis of $L \otimes_{E} N'$,
there exists a matrix $B \in \Mat_{s'' \times s'}(L)$ such that
$$\tilde{\mbn}'' = B \mbx' + \Psi_{N''} A'' \tilde{\mbx}''.$$
Therefore we have
$$\Psi_{N} A \mbx = \mbn = \begin{bmatrix} \mbn' \\ \tilde{\mbn}'' \end{bmatrix}
= \begin{bmatrix} \Psi_{N'} A' & 0 \\ B & \Psi_{N''} A'' \\ \end{bmatrix} \mbx.$$
Since $\Psi_{N} \in \GL_{s}(\Sigma)$,
we conclude that $\Psi_{N'} \in \GL_{s'}(\Sigma)$ and $\Psi_{N''} \in \GL_{s''}(\Sigma)$.
\end{proof}

\begin{lemm}\label{p-d_isom}
For any $N \in \TM$ and $F$-algebra $R$, there exists a natural isomorphism
$$\Sigma\RR \otimes_{F} V(N) \rar \Sigma\RR \otimes_{E} N.$$
Similarly, there exists a natural isomorphism
$$\Sl\RR \otimes_{F} V(N) \rar \Sl\RR \otimes_{E} N$$
for all $l$.
\end{lemm}

\begin{proof}
The inclusion $V(N) \subset \Sigma \otimes_{E} N$
and the product map $\Sigma \otimes_{F} \Sigma \rar \Sigma$ induce a natural map
$$\kappa : \Sigma\RR \otimes_{F} V(N) \hookrightarrow \Sigma\RR \otimes_{F} \Sigma \otimes_{E} N \rar \Sigma\RR \otimes_{E} N.$$
Since $1 \otimes \Psi_{N}^{-1} \mbn$ is a $\Sigma\RR$-basis of $\Sigma\RR \otimes_{F} V(N)$,
we can write $\kappa$ explicitly as follows:
$$\kappa(\mbf \cdot (1 \otimes \Psi_{N}^{-1} \mbn)) = (\mbf \Psi_{N}^{-1}) \cdot (1 \otimes \mbn)$$
for all $\mbf \in \Mat_{1 \times s}(\Sigma\RR)$.
Hence it is clear that $\kappa$ is an isomorphism.
The $\Sl$ version is proved by the same argument.
\end{proof}

\begin{theo}\label{p-d_action}
For any $N \in \TM$, there exists a natural representation
$$\rho_{N} : \Gamma \rar \GL(V(N))$$
over $F$ that is functorial in $N$.
\end{theo}

\begin{proof}
For any $F$-algebra $R$ and $\gamma \in \Gamma(R) \subset \GL_{r}(R)$, we define
$$\rho_{N}\RR(\gamma) : R \otimes_{F} V(N) \hookrightarrow \Sigma\RR \otimes_{F} V(N)
\rar \Sigma\RR \otimes_{E} N \rar \Sigma\RR \otimes_{E} N,$$
where the second map is the isomorphism defined in Lemma \ref{p-d_isom}
and the third map is defined by $h(\Psi) \otimes x \mapsto h(\Psi \gamma) \otimes x$.
Clearly $\rho_{N}\RR$ is functorial in $N$.
If $\im(\rho_{M}\RR(\gamma)) = R \otimes_{F} V(M)$
then $\im(\rho_{N}\RR(\gamma)) = R \otimes_{F} V(N)$ for all $N \in \TM$.
Thus we may assume that $N = M$.
We can write $\rho_{M}\RR(\gamma)$ explicitly:
$$\rho_{M}\RR(\gamma)(\mbf \cdot (1 \otimes \Psi^{-1} \mbm)) = \mbf \gamma^{-1} (1 \otimes \Psi^{-1} \mbm),$$
for each $\mbf \in \Mat_{1 \times r}(R)$.
Therefore we have $\im(\rho_{M}\RR(\gamma)) = R \otimes_{F} V(M)$.
\end{proof}

From the above description of $\rho_{M}\RR$, we have the following corollary:
\begin{coro}\label{p-d_faithful}
The representation $\rho_{M} : \Gamma \rar \GL(V(M))$ is faithful.
\end{coro}

From Theorem \ref{p-d_action}, we have a functor $\xi_{M} : \TM \rar \Rep(\Gamma, F)$,
and it is clear by the construction that $\xi_{M}$ is a tensor functor.
Let $\eta_{M} : \Rep(\Gamma_{M}, F)) \rar \TM$ be
the equivalence of categories defined by the Tannakian duality
and $\alpha : \Rep(\Gamma, F) \rar \vsp(F)$ the forgetful functor.
Since $V_{M} = \alpha \circ \xi_{M}$,
there exists a unique homomorphism $\pi_{M} : \Gamma \rar \Gamma_{M}$ over $F$ such that
the natural functor $\tau_{M} : \Rep(\Gamma_{M}, F) \rar \Rep(\Gamma, F)$ induced by $\pi_{M}$
satisfies $\xi_{M} \circ \eta_{M} = \tau_{M}$.
\[\xymatrix{
\Rep(\Gamma_{M}, F) \ar[r]^>>>>>{\eta_{M}} & \TM \ar[r]^>>>>>{\xi_{M}} \ar[dr]_{V_{M}} & \Rep(\Gamma, F) \ar[d]^{\alpha} \\
& & \vsp(F) \\
}\]

\begin{prop}\label{p-d_subq}
For any representation $W \in \Rep(\Gamma, F)$,
there exists an object $N \in \TM$ such that $W$ is isomorphic to a subquotient of $\xi_{M}(N)$.
\end{prop}

\begin{proof}
By Corollary \ref{p-d_faithful}, the $\Gamma$-representation $\xi_{M}(M) = \rho_{M}$ is faithful.
Therefore, $W$ is isomorphic to a subquotient of representation of the form
$$\oplus_{i=1}^{n} (\xi_{M}(M))^{\otimes a_i} \otimes (\xi_{M}(M)^{\vee})^{\otimes b_i},$$
where $a_{i}, b_{i} \in \mathbb{N}$.
However we have
$\oplus_{i=1}^{n} (\xi_{M}(M))^{\otimes a_i} \otimes (\xi_{M}(M)^{\vee})^{\otimes b_i}
= \xi_{M}(\oplus_{i=1}^{n} M^{\otimes a_i} \otimes (M^{\vee})^{\otimes b_i})$.
\end{proof}

Proposition \ref{p-d_subq} is equivalent to the next theorem (\cite{DeMi}, Proposition 2.21).
\begin{theo}\label{p-d_closed}
The morphism of affine $F$-schemes $\pi_{M} : \Gamma \rar \Gamma_{M}$ is a closed immersion.
\end{theo}

From now on, we assume that
$\Gamma(F)$ is Zariski dense in $\Gamma$ or $\Lal / F$ is a regular extension for each $l$.
In the former case we put $F' = F$, and in the latter case we put $F' = \bar{F}$.
For any $F$-algebra $S$, we set $S' := F' \otimes_{F} S$.
Then in any case, $E'$ and $\Lal'$ are fields,
$\Lal' = \Frac(\Sl')$ and $\La \cap (\La')^{\Gamma(F')} = E$ by Theorem \ref{gal_gammafix}.

\begin{prop}\label{p-d_f.f.}
Assume that $\Gamma(F)$ is Zariski dense in $\Gamma$ or $\Lal / F$ is a regular extension for each $l$.
Then the functor $\xi_{M} : \TM \rar \Rep(\Gamma, F)$ is fully faithful.
\end{prop}

\begin{proof}
For any objects $N, N' \in \TM$, there exist natural isomorphisms
$\Hom_{\TM}(N', N) \cong \Hom_{\TM}(\mathbf{1}, \Hom(N', N))$
and $\Hom_{\Gamma}(V(N'), V(N)) \cong \Hom_{\Gamma}(V(\mathbf{1}), V(\Hom(N', N)))$.
Thus it is enough to show that, for any $N \in \TM$,
$\Hom_{\TM}(\mathbf{1}, N) \rar \Hom_{\Gamma}(V(\mathbf{1}), V(N))$ is an isomorphism.
It is injective since
$\Hom_{\TM}(\mathbf{1}, N) = N^{\varphi} = N \cap V(N) \hookrightarrow \Hom_{\Gamma}(V(\mathbf{1}), V(N))$.
For any $\phi \in \Hom_{\Gamma}(V(\mathbf{1}), V(N))$,
there exists $\mbh = \mbh(\Psi) \in \Mat_{1 \times s}(\Sigma)$
so that $\phi(1) = \mbh \mbn$ by Lemma \ref{p-d_isom}.
Then for any $\gamma \in \Gamma(F')$,
we have $\mbh(\Psi) \mbn = \phi(1) = \gamma . \phi(1) = \mbh(\Psi \gamma) \mbn$.
Hence $\mbh(\Psi) = \mbh(\Psi \gamma) = \gamma . \mbh$.
By Theorem \ref{gal_gammafix}, we have $\mbh \in \Mat_{1 \times s}(E)$,
and this implies $\phi(1) = \mbh \mbn \in N \cap V(N)$.
\end{proof}

We prepare a lemma from linear algebra.
\begin{lemm}\label{p-d_ope}
Let $E \subset \La$ be general rings where $E$ is a field
and $\La = \prod_{l \in \mathbb{Z}/d'} \Lal$ is a finite product of fields.
Assume that $\# E > d'$.
Let $1 \leq m \leq s$ and $D \in \Mat_{s \times m}(\La)$.
If there exists $D_{0} \in \GL_{s}(\La)$ such that $D_{0} = [*, D]$,
then there exist $A \in \GL_{m}(\La)$ and $B \in \GL_{s}(E)$ such that
$$BDA =
\begin{bmatrix}
1 & & \\
& \ddots &\\
& & 1 \\
* & * & * \\
\end{bmatrix}
\in \Mat_{s \times m}(\La).
$$
\end{lemm}

\begin{proof}
For each $l \in \mathbb{Z}/d'$ and $1 \leq j \leq m$, let $e_{l,j} \in \Mat_{1 \times m}(\Lal)$ be
a row vector such that the $j$-th component is one and the other components are zero.
Write $D = (D_{l})_{l}$ where $D_{l} \in \Mat_{s \times m}(\Lal)$.
Since the rank of $D_{l}$ is $m$ for each $l$,
there exists a matrix $\tilde{A}_{l} \in \GL_{m}(\Lal)$ such that
$$D_{l} \tilde{A}_{l} =
\begin{bmatrix}
C_{l,1}\\
\vdots\\
C_{l,s}\\
\end{bmatrix},
$$
where $C_{l,i} \in \Mat_{1 \times m}(\Lal)$,
and for each $1 \leq j \leq m$ there exists an $i$ such that $C_{l,i} = e_{l,j}$.
An elementary pattern of $D_{l} \tilde{A}_{l}$ is a choice of
$(i_{l1},\dots,i_{lm}) \in \{1,\dots,s\}^{m}$ such that $C_{l,i_{l k}} = e_{l,k}$ for each $1 \leq k \leq m$.
We fix an elementary pattern $(i_{l1},\dots,i_{lm})$ of $D_{l} \tilde{A}_{l}$ for each $l$.
For each matrix $P \in \Mat_{s \times m}(\Lal)$ such that,
for each $1 \leq j \leq m$ there exists an $i$ such that the $i$-th row of $P$ is $e_{l,j}$,
we define an elementary pattern of $P$ in the same way.
For a matrix in $\Mat_{s \times m}(\La)$, we define the procedures
\begin{description}
\item[$(1)$] left-multiplication by a matrix in $\GL_{s}(E)$,
\item[$(2_{l})$] right-multiplication by a matrix in $\GL_{m}(\Lal) \times \prod_{l' \neq l} \{1\}$.
\end{description}
Set $\tilde{A} := (\tilde{A}_{l})_{l} \in \GL_{m}(\La)$ and $C_{i} := (C_{l,i})_{l} \in \Mat_{1 \times m}(\La)$.
By using the above procedures,
we want to transform $D \tilde{A}$ to a matrix $D' = (D'_{l})_{l}$ such that,
we can choose an elementary pattern of $D'_{l}$ to $(1,\dots,m)$ for each $l$.

Fix $i' \neq i''$ and $l_{0}$.
Let $\tau = (i' \ i'')$ be the transposition of $i'$ and $i''$.
It is enough to show that,
by using the procedures $(1)$ and $(2_{l})$,
we can transform $D \tilde{A}$ to a matrix $D' = (D'_{l})_{l}$ such that,
we can choose an elementary pattern of $D'_{l_0}$ to $(\tau i_{l_{0}1},\dots,\tau i_{l_{0}m})$
and an elementary pattern of $D'_{l}$ to $(i_{l1},\dots,i_{lm})$ for each $l \neq l_{0}$.

First we assume that $i' = i_{l_{0} j'}$ and $i'' = i_{l_{0} j''}$ for some $j' \neq j''$.
For $c \in E^{\times}$, we can exchange the $i'$-th row of $D \tilde{A}$ for $C_{i'} + c C_{i''}$ by the procedure $(1)$.
Since $\# E > d'$, we can take $c$ such that, for each $l \neq l_{0}$,
if $i' = i_{l j}$ for some $j$ then the $j$-th component of $C_{l,i'} + c C_{l,i''}$ is non-zero.
Then by the procedures $(2_{l})$ for $l \neq l_{0}$, we can transform this matrix to a matrix $D'' = (D''_{l})_{l}$ such that,
we can choose an elementary pattern of $D''_{l}$ to $(i_{l1},\dots,i_{lm})$ for each $l \neq l_{0}$,
the $i$-th row of $D''_{l_{0}}$ is $C_{l_{0},i}$ for each $i \neq i'$ and the $i'$-th row of $D''_{l_{0}}$ is
$$
(0,\dots,0,\stackrel{j'}{\stackrel{\vee}{1}},0,\dots,0,\stackrel{j''}{\stackrel{\vee}{c}},0,\dots,0).
$$
Therefore by the procedure $(2_{l_{0}})$, we can transform $D''$ to a matrix $D'$
which has the desired properties.
The case that $i' \not\in \{i_{l_{0} 1},\dots,i_{l_{0} m}\}$ and $i'' = i_{l_{0} j''}$ for some $j''$
is proved in a similar way, and we omit the proof.
\end{proof}

\begin{lemm}\label{p-d_def-field}
Assume that $\Gamma(F)$ is Zariski dense in $\Gamma$ or $\Lal / F$ is a regular extension for each $l$.
Assume also that $\# E > d'$.
We take $1 \leq m \leq s$ and $D \in \Mat_{s \times m}(\La)$
such that $[*, D] \in \GL_{s}(\La)$ for some $* \in \Mat_{s \times (s-m)}(\La)$.
We set
$$W := \{\mbx \in \Mat_{1 \times s}(\La') | \mbx D = 0\},$$
and assume that $\Gamma(F') W \subset W$,
where the elements of $\Gamma(F')$ act on $W$ by componentwise.
Then there exists a matrix $C \in \Mat_{(s-m) \times s}(E)$
such that the rank of $C$ is $s-m$ and $CD = 0$.
\end{lemm}

\begin{proof}
By Lemma \ref{p-d_ope}, there exist matrices $A \in \GL_{m}(\La)$ and $B \in \GL_{s}(E)$ such that
$$B D A =
\begin{bmatrix}
I_{m}\\
C_{0}\\
\end{bmatrix},
$$
where $I_{m}$ is the identity matrix of size $m$ and $C_{0} \in \Mat_{(s-m) \times m}(\La)$.
We set
$$W_{B} := W B^{-1} = \{\mbx \in \Mat_{1 \times s}(\La') | \mbx B D = 0\}
= \{\mbx \in \Mat_{1 \times s}(\La') | \mbx \begin{bmatrix} I_{m}\\ C_{0}\\ \end{bmatrix} = 0\}.$$
Then it is clear that $W_{B}$ is also $\Gamma(F')$-stable.
Thus, since each row of $\begin{bmatrix} -C_{0} & I_{s-m}\\ \end{bmatrix}$ is an element of $W_{B}$,
each row of $\begin{bmatrix} -\gamma C_{0} & I_{s-m}\\ \end{bmatrix}$ is also an element of $W_{B}$
for any $\gamma \in \Gamma(F')$.
This means that $\gamma C_{0} = C_{0}$ for each $\gamma \in \Gamma(F')$.
Therefore $C_{0} \in \Mat_{(s-m) \times m}(E)$  by Theorem \ref{gal_gammafix}.
We set $C := \begin{bmatrix} -C_{0} & I_{s-m}\\ \end{bmatrix} B$.
Then it is clear that this $C$ has the desired properties.
\end{proof}

\begin{prop}\label{p-d_sub}
Assume that $\Gamma(F)$ is Zariski dense in $\Gamma$
or $\Lal / F$ is a regular extension for each $l$.
Assume also that $\# E > d'$.
For any $N \in \TM$ and $\Gamma$-subrepresentation $U \subset \xi_{M}(N)$,
there exists a $\varphi$-submodule $N' \subset N$ such that $\xi_{M}(N') = N$.
\end{prop}

\begin{proof}
We take $\mbu \in \Mat_{u \times 1}(U)$ an $F$-basis of $U$
such that $\bar{\mbn} := \begin{bmatrix} \mbu & *\\ \end{bmatrix}^{\mathrm{tr}}$ forms an $F$-basis of $\xi_{M}(N)$.
By Lemma \ref{p-d_isom}, we have $\bar{\mbn} = H \mbn$ for some $H = H(\Psi) \in \GL_{s}(\Sigma)$.
We take a matrix $D \in \Mat_{s \times (s-u)}(\Sigma)$ such that $H^{-1} = \begin{bmatrix} * & D\\ \end{bmatrix}$,
and set $W := \{\mbx \in \Mat_{1 \times s}(\La') | \mbx D = 0\}$.
Since $I_{s} = H H^{-1} = \begin{bmatrix} * & H D\\ \end{bmatrix}$,
the $i$-th row of $H$ is an element of $W$ for each $i \leq u$.
These form a $\La'$-basis of $W$ because the coefficient ring $\La'$ is a finite product of fields.
For each $\gamma \in \Gamma(F')$,
we have $\gamma \bar{\mbn} = (\gamma H) \mbn = (\gamma H) H^{-1} \bar{\mbn}$.
Since $U$ is $\Gamma$-stable, the $(i,j)$-th component of
$(\gamma H) H^{-1} = \begin{bmatrix} * & (\gamma H) D\\ \end{bmatrix}$ is zero for each $i \leq u$ and $j > u$.
Therefore, $W$ is $\Gamma(F')$-stable.
By Lemma \ref{p-d_def-field}, there exists a matrix $C \in \Mat_{u \times s}(E)$
such that the rank of $C$ is $u$ and $C D = 0$.
Then we can take $B \in \GL_{s}(E)$ such that $C$ forms the top rows of $B$.
Let $\begin{bmatrix} \mbn' & \mbn''\\ \end{bmatrix}^{\mathrm{tr}} := B \mbn$
where $\mbn' \in \Mat_{u \times 1}(N)$.
Let
$$B H^{-1} = \begin{bmatrix} C\\ *\\ \end{bmatrix} \begin{bmatrix} * & D\\ \end{bmatrix}
=: \begin{bmatrix} \Psi' & 0\\ * & *\\ \end{bmatrix},$$
where $\Psi' \in \GL_{u}(\Sigma)$.
Then we have
\begin{align*}
\varphi \begin{bmatrix} \mbn'\\ \mbn''\\ \end{bmatrix}
&= \varphi(B \mbn) = \varphi(B H^{-1} H \mbn) = \sigma(B H^{-1}) \varphi(H \mbn) = \sigma(B H^{-1}) H \mbn \\
&= \sigma(B H^{-1}) H B^{-1} \begin{bmatrix} \mbn'\\ \mbn''\\ \end{bmatrix} =: \begin{bmatrix} \Phi' & 0\\ * & *\\ \end{bmatrix}
\begin{bmatrix} \mbn'\\ \mbn''\\ \end{bmatrix},
\end{align*}
where $\Phi' \in \GL_{u}(E)$.
Hence $N' := \langle \mbn' \rangle_{E} \subset N$ is a sub $\varphi$-module,
and we have $\varphi \mbn' = \Phi' \mbn'$.
Moreover, we have
$$\sigma \begin{bmatrix} \Psi' & 0\\ * & *\\ \end{bmatrix} = \sigma(B H^{-1}) = (\sigma(B H^{-1}) H B^{-1}) (B H^{-1}) 
= \begin{bmatrix} \Phi' & 0\\ * & *\\ \end{bmatrix} \begin{bmatrix} \Psi' & 0\\ * & *\\ \end{bmatrix}
= \begin{bmatrix} \Phi' \Psi' & 0\\ * & *\\ \end{bmatrix}.$$
Therefore, $\Psi'$ is a fundamental matrix for $\Phi'$.
Since
$$\begin{bmatrix} \mbn'\\ \mbn''\\ \end{bmatrix} = B \mbn = B H^{-1} \bar{\mbn}
= \begin{bmatrix} \Psi' & 0\\ * & *\\ \end{bmatrix} \begin{bmatrix} \mbu\\ *\\ \end{bmatrix},$$
we have that $\xi_{M}(N') = \langle (\Psi')^{-1} \mbn' \rangle_{F} = \langle \mbu \rangle_{F} = U$.
\end{proof}

\begin{theo}\label{p-d_pi-isom}
Assume that $\Gamma(F)$ is Zariski dense in $\Gamma$ or $\Lal / F$ is a regular extension for each $l$.
Assume also that $\# E > d'$.
Then the morphism of affine $F$-schemes $\pi_{M} : \Gamma \rar \Gamma_{M}$ is an isomorphism.
Equivalently, the functor $\xi_{M} : \TM \rar \Rep(\Gamma, F)$ is an equivalence of Tannakian categories.
\end{theo}

\begin{proof}
By Propositions \ref{p-d_f.f.} and \ref{p-d_sub}, $\pi_{M}$ is faithfully flat (\cite{DeMi}, Proposition 2.21).
On the other hand, $\pi_{M}$ is a closed immersion by Theorem \ref{p-d_closed}.
Therefore $\pi_{M}$ is an isomorphism.
\end{proof}

\subsection{$v$-adic case}
In this subsection, we continue to use the notations of the previous subsection
and consider the case that $(F, E, L) = (\ftv, \ktv, \kstv)$,
where the notations are defined in Subsection \ref{phi_v}.

The assumption that $\Lal / F$ is regular for each $l$ is not true in general.
For example, assume that $r = 1$, $v = t$ and $\Phi \in K$ such that $\Psi := \Phi^{1/(q-1)} \not\in \ktv$.
Then $\Psi$ is a fundamental matrix for $\Phi$, and it is clear that $Z$ is not absolutely irreducible.
Therefore the assumptions are not satisfied.
However we expect that this assumption is true for \lq\lq good'' objects.

Hence we consider the other assumption.
In the $v$-adic case, $\Gamma(\ftv)$ contains a Galois image.
Since the Galois image is large enough, we can conclude that $\Gamma(\ftv)$ is Zariski dense in $\Gamma$.

\begin{lemm}\label{p-d_closure}
Let $G$ be an algebraic group over a field $k$ and $H$ a subgroup of $G(k)$.
We set $H^{Zar}$ the Zariski closure of $H$ in $G$ endowed with the reduced structure.
Then $H^{Zar}$ is a subgroup scheme of $G$ and smooth.
\end{lemm}

\begin{proof}
We denote by $\bar{H}$ the Zariski closure of $H$ in $G(\bar{k})$.
By Lemma \ref{gal_closure}, $\bar{H}$ is defined over $k$.
Then it is clear that $(H^{Zar})_{\bar{k}} = \bar{H}$.
Thus $H^{Zar}$ is absolutely reduced.

To prove that $H^{Zar}$ is a group scheme, it is enough to show that $\bar{H}$ is a group.
For any $a \in G(\bar{k})$, the map $G(\bar{k}) \rar G(\bar{k}) ; g \mapsto a g$ is a homeomorphism.
Thus for any $a \in H$, we have $a \bar{H} = \overline{a H} \subset \bar{H}$.
Thus for any $b \in \bar{H}$, we have $H b \subset \bar{H}$.
Therefore $\bar{H} b = \overline{H b} \subset \bar{H}$.
Hence we have $\bar{H} \bar{H} \subset \bar{H}$.
Since the map $G(\bar{k}) \rar G(\bar{k}) ; g \mapsto g^{-1}$ is a homeomorphism,
we have $\bar{H}^{-1} = \overline{H^{-1}} = \bar{H}$.
\end{proof}

\begin{lemm}\label{p-d_image}
Let $G$ be a topological group and $k$ be a topological field.
Let $\rho : G \rar \GL_{r}(k)$ be a continuous $k$-representation of $G$.
We set $\mathcal{C}_{\rho}$ the Tannakian subcategory of $\Rep(G, k)$ generated by $\rho$
and $\Gamma_{\rho} \subset \GL_{r/k}$ its Tannakian Galois group.
Then $\rho(G)$ is Zariski dense in $\Gamma_{\rho}$.
\end{lemm}

Note that the Tannakian Galois group $\Gamma_{\rho}$ may not be reduced.

\begin{proof}
We have an inclusion $\rho(G)^{Zar} \subset \Gamma_{\rho}$ and $\rho$ factors through a $\rho(G)^{Zar}(k)$:
$$\rho : G \rar \rho(G)^{Zar}(k) \hookrightarrow \Gamma_{\rho}(k) \hookrightarrow \GL_{r}(k).$$
Thus we have functors of Tannakian categories
$$\mathcal{C}_{\rho} \cong \Rep(\Gamma_{\rho}, k) \rar \Rep(\rho(G)^{Zar}, k) \rar \Rep(G, k).$$
We denote by $\Gamma_{G,k}$ be the Tannakian Galois group of $\Rep(G, k)$.
Then we have morphisms of algebraic groups which correspond to the above sequence:
$$\Gamma_{G,k} \rar \rho(G)^{Zar} \hookrightarrow \Gamma_{\rho}.$$
Since $\Gamma_{G,k} \rar \Gamma_{\rho}$ is an epimorphism of algebraic groups,
we have $\rho(G)^{Zar}(\bar{k}) = \Gamma_{\rho}(\bar{k})$.
\end{proof}

For any $\tau \in G_{K}$, since $\sigma(\tau \Psi) = \tau(\sigma \Psi) = \tau(\Phi \Psi) = \Phi(\tau \Psi)$,
there exists a matrix $A_{\tau} \in \GL_{r}(\ftv)$ such that $\tau \Psi = \Psi A_{\tau}$.
Therefore we have $\tau(\Sigma) = \Sigma$ and a map $G_{K} \rar \aus(\Sigma / \ktv)$.
By Lemma \ref{gal_g=a}, we have that $A_{\tau} \in \Gamma(\ftv)$ and $A_{\tau}$
corresponds to the image of $\tau$ in $\aus(\Sigma / \ktv)$
via the isomorphism $\aus(\Sigma / \ktv) \cong \Gamma(\ftv)$.

On the other hand, we can verify that the map
$$G_{K} \rar \aus(\Sigma / \ktv) \cong \Gamma(\ftv) \hookrightarrow \Gamma_{M}(\ftv) \hookrightarrow \GL(V(M))$$
coincide with the natural representation $G_{K} \rar \GL(V(M))$ defined in Subsection \ref{phi_v}.

\begin{prop}\label{p-d_gal-im}
The image of $G_{K}$ in $\Gamma_{M}(\ftv)$ is Zariski dense in $\Gamma_{M}$.
\end{prop}

\begin{proof}
Let $\mathcal{C}_{M}$ be the Tannakian subcategory of $\Rep(G_{K}, \ftv)$ generated by $V(M)$.
Then by Theorem \ref{phi_main-equiv}, the categories $\TM$ and $\mathcal{C}_{M}$ are equivalence.
Therefore $\Gamma_{M}$ is also a Tannakian Galois group of $\mathcal{C}_{M}$.
Hence by Lemma \ref{p-d_image}, the image of $G_{K}$ is Zariski dense in $\Gamma_{M}$.
\end{proof}

\begin{theo}\label{p-d_v-isom}
If $(F, E, L) = (\ftv, \ktv, \kstv)$, then the morphism $\pi_{M} : \Gamma \rar \Gamma_{M}$ is an isomorphism.
\end{theo}

\begin{proof}
By Proposition \ref{p-d_gal-im}, $\Gamma(\ftv)$ is Zariski dense in $\Gamma_{M}$.
In particular, it is Zariski dense in $\Gamma$.
Therefore by Theorem \ref{p-d_pi-isom}, $\pi_{M}$ is an isomorphism.
\end{proof}

\begin{prop}\label{p-d_f-val}
Fix an index $m \in \mathbb{Z}/d$ and take an element $\tau \in G_{K}$
such that $\tau|_{\mathbb{F}_{q^d}} = \sigma|_{\mathbb{F}_{q^d}}^{-m}$.
Then the image of $\tau$ in $\Gamma(\ftv)$ is contained in $\Gamma_{m}(\ftv)$.
\end{prop}

\begin{proof}
Since $\tau$ induces a $\ktv$-isomorphism $K^{\mathrm{sep}}(\!(t-\lambda_{l+m})\!) \rar \kstl$,
also induces a bijection $Z_{l+m}(K^{\mathrm{sep}}(\!(t-\lambda_{l+m})\!)) \rar Z_{l+m}(\kstl)$.
Let $A_{\tau} \in \Gamma(\ftv)$ be as above.
Since $\Psi A_{\tau} = \tau \Psi$, we have $\Psi_{l} A_{\tau} = \tau \Psi_{l+m} \in Z_{l+m}(\kstl)$.
Note that $\Psi_{l} \in Z_{l}(\kstl)$ for each $l$ by the definition of $Z_{l}$.
Therefore by Theorem \ref{gal_tri2}, we have
$A_{\tau} \in \Gamma_{m}(K^{\mathrm{sep}}(\!(t-\lambda_{l})\!)) \cap \Gamma(\ftv) = \Gamma_{m}(\ftv)$.
\end{proof}

\section{$v$-adic criterion}\label{sec_abp}
In this section, we set $K := \Fq(\theta)$
the rational function field over $\Fq$ with one variable $\theta$ independent of $t$.
Let $M$ be a finite-dimensional $\vp$-module over $\ktv$,
$\mbm$ a $\ktv$-basis of $M$
and $\Phi \in \Mat_{r \times r}(\ktv)$ a matrix such that $\vp \mbm = \Phi \mbm$.

\begin{defi}
A $\vp$-module $M$ is said to be a \textit{v-adic t-motive} if $\Phi \in \Mat_{r \times r}(K[t]_{v})$
and $\det \Phi = c (t - \theta)^{s}$ for some $c \in \bar{K}^{\times}$ and $s \in \mathbb{N}$.
\end{defi}

Since $t - \theta$ is invertible in $K[t]_{v}$,
$v$-adic $t$-motives are $\kstv$-trivial by Theorem \ref{phi_main-equiv}.
Thus we can apply the results of the previous sections to $v$-adic $t$-motives.

\begin{rema}
Let $k$ be a field of characteristic $p > 0$
and $\iota : \mathbb{F}_{p}[t] \rar k$ a ring homomorphism.
Anderson defined the notion of $t$-motives over $k$ in \cite{Ande}.
This is a $\vp$-module $M$ over $k[t]$ which satisfies the following conditions:
\begin{itemize}
\item $M$ is free of finite rank over $k[t]$.
\item $(t-\iota(t))^{N}(M/k[t] \cdot \vp M) = 0$ for some integer $N > 0$.
\item $M$ is finitely generated over $k_{\sigma}[\vp]$.
\end{itemize}
Here the $\vp$ action on $k[t]$ is defined as before
and $k_{\sigma}[\vp]$ is the subring of $k[t]_{\sigma}[\vp]$ generated by $k$ and $\vp$.
Thus we have a functor from the category of $t$-motives over $K$ (here we take $\iota(t) = \theta$)
to the category of $v$-adic $t$-motives by tensoring $\ktv$.
\end{rema}

Let $\kv$ be the completion of $K$ with respect to the place at $v(\theta)$,
$K_{d} := K \cdot \mathbb{F}_{q^d}$
the composite field of $K$ and $\mathbb{F}_{q^d}$ in $\bar{K}$,
$\kll := K_{d,(\theta - \lambda_{l})}$
the completion of $K_{d}$ with respect to the place at $(\theta-\lambda_{l})$,
$\overline{\kll}$ an algebraic closure of $\kll$,
and $\cll := \widehat{\overline{\kll}}$
the completion of $\overline{\kll}$ with respect to the canonical extension of $(\theta-\lambda_{l})$.
Let $v_{l}$ be the valuation on $\cll$ normalized by $v_{l}(\theta-\lambda_{l}) = 1$.
For each $l$, we fix an embedding $\bar{K}$ to $\overline{\kll}$ over $K_{d}$.
Then for each $f \in \kstv = \prod_{l} \kstl$,
we can define $f(\theta) \in \prod_{l} \cll$ by substituting $\theta$ for $t$ if it converge.
We have the following conjecture, which is a $v$-adic analogue of Proposition 3.1.1 in \cite{ABP}:

\begin{conj}\label{abp_conj}
Let $\Phi \in \GL_{r}(\ktv) \cap \Mat_{r \times r}(K[t])$ and
$\psi \in \Mat_{r \times 1}(K^{\mathrm{sep}}[t]_{v})$ be matrices such that
$\psi(\theta)$ converges, $\sigma \psi = \Phi \psi$
and $\det \Phi = c (t-\theta)^{s}$ for some $c \in K^{\times}$ and $s \in \mathbb{N}$.
Then, any linear relation of the components of $\psi(\theta)$ over $K_{v(\theta)}$
lifts to some linear relation of the components of $\psi$ over $K[t]_{v}$.
Precisely speaking,
if there exists an element $\rho \in \Mat_{1 \times r}(K_{v(\theta)})$ such that $\rho \psi(\theta) = 0$,
then there exists an element $P \in \Mat_{1 \times r}(K[t]_{v})$ such that
$P \psi = 0$, $P(\theta)$ converges and $P(\theta) = \rho$.
\end{conj}

Conjecture \ref{abp_conj} is true if $r = 1$ and we give a proof below.
This proof is the same as the proof of the $\infty$-adic version for $r = 1$ in \cite{ABP}.

If $\rho = 0$, then we can take $P = 0$.
Therefore we may assume that $\rho \neq 0$.
Since for some $P \in K[t]_{v}$, we have $P(\theta) = \rho$.
Hence it is enough to show that, if $\psi(\theta) = 0$, then $\psi = 0$.
Write $\psi = (\sum_{i} a_{l,i} (t-\lambda_{l})^{i})_{l}$.
For any $\nu \geq 0$,
the infinite sum $\sum_{i} a_{l,i} ((\theta-\lambda_{l})^{q^{d \nu}})^{i}$ converge
because $\sum_{i} a_{l,i} (\theta-\lambda_{l})^{i}$ converges
and $v_{l}((\theta-\lambda_{l})^{q^{d \nu}}) \geq v_{l}(\theta-\lambda_{l})$ for each $l$.
Thus we have
$$\psi(\theta^{q^{d \nu}})^{q^d} = (\sum_{i} a_{l,i} ((\theta-\lambda_{l})^{q^{d \nu}})^{i})_{l}^{q^d}
= (\sum_{i} a_{l,i}^{q^d} (\theta^{q^{d(\nu + 1)}}-\lambda_{l})^{i})_{l} = (\sigma^{d}\psi)(\theta^{q^{d(\nu + 1)}}).$$
On the other hand, we have
\begin{align*}
(\sigma^{d}\psi)(\theta^{q^{d(\nu + 1)}})
&= (\sigma^{d-1}\Phi)(\theta^{q^{d(\nu + 1)}}) \times \cdots \times (\sigma^{0}\Phi)(\theta^{q^{d(\nu + 1)}}) \times \psi(\theta^{q^{d(\nu + 1)}}) \\
&= c^{q^{d-1} + \cdots + q^{0}} (\theta^{q^{d(\nu + 1)}} - \theta^{q^{d-1}})^{s} \cdots (\theta^{q^{d(\nu + 1)}} - \theta^{q^{0}})^{s} \psi(\theta^{q^{d(\nu + 1)}}).
\end{align*}
By induction on $\nu$,
we have $\sum_{i} a_{l,i}((\theta - \lambda_{l})^{q^{d \nu}})^{i} = 0$ for each $l$.
Thus the formal series $\sum_{i} a_{l,i} z^{i}$ has infinite zeros on the disk
$v_{l}(z) \geq v_{l}(\theta - \lambda_{l})$.
Therefore $a_{l,i} = 0$ for all $l$ and $i$, and we conclude that $\psi = 0$.

Next, we calculate valuations of the coefficients of periods for some examples of $t$-motives.
An element $L_{\alpha,n}$ is an analogue of the $n$-th Carlitz polylogarithm,
and an element $\Omega_{v}$ is an analogue of the Carlitz period.

\begin{prop}\label{abp_l}
Let $n \geq 1$ be an integer and
$\alpha \in (\ks)^{\times}$ an element such that $v_{l}(\alpha) \geq 0$ for all $l$.
Then there exists an element
$L_{\alpha,n} = L_{\alpha,n}(t) = (\sum_{i} a_{l,i} (t-\lambda_{l}))_{l} \in \ks[t]_{v} = \prod_{l}\ks[\![t-\lambda_{l}]\!]$
which satisfies the equation
$$\sigma(L_{\alpha,n}) = \sigma(\alpha) + L_{\alpha,n}/(t-\theta)^{n}.$$
For any $l \in \mathbb{Z}/d$, $0 \leq m \leq d-1$ and $i \geq 0$, we have
$$v_{l}(a_{l+m,i}) \geq -q^{m} \left( \frac{i}{q^{d}} + \frac{n}{q^{d}-1} \right).$$
\end{prop}

\begin{proof}
For an element $L_{\alpha,n} = (\sum_{i} a_{l,i} (t-\lambda_{l}))_{l} \in \prod_{l}\kstl$,
we have an explicit descriptions
$$(t-\theta)^{n} \sigma(L_{\alpha,n}) =
\left(\sum_{i} \left(\sum_{j=0}^{n} \binom{n}{j} (\lambda_{l}-\theta)^{n-j} a_{l-1,i-j}^{q}\right) (t-\lambda_{l})^{i}\right)_{l},$$
$$\sigma(\alpha) (t-\theta)^{n} =
\left(\sum_{i=0}^{n} \binom{n}{i} (\lambda_{l}-\theta)^{n-i} \alpha^{q} (t-\lambda_{l})^{i}\right)_{l}.$$
We set $b_{l,i} :=
\sum_{j=1}^{n} \binom{n}{j}(\lambda_{l}-\theta)^{n-j}a_{l-1,i-j}^{q} - \binom{n}{i}(\lambda_{l}-\theta)^{n-i}\alpha^{q}$
and $c_{l} := (\lambda_{l}-\theta)^{n}$.
Then the equation in Proposition is equivalent to the equations
$$a_{l+1,i} = c_{l+1} a_{l,i}^{q} + b_{l+1,i}$$
for all $l \in \mathbb{Z}/d$ and $i \in \mathbb{Z}$.
For $i < 0$, we can take $a_{l,i} = 0$.
Fix $i \geq 0$ and consider the system of polynomial equations
$$X_{l+1} = c_{l+1} X_{l}^{q} + b_{l+1,i} \ \ (l \in \mathbb{Z}/d).$$
For $2 \leq r \leq m$, we set
$$\beta_{m,r,i} := b_{r,i}^{q^{m-r}} \prod_{s=r+1}^{m} c_{s}^{q^{m-s}} \textrm{ and }
\gamma_{m} := \prod_{s=2}^{m} c_{s}^{q^{m-s}}.$$
Then the above equations are equivalent to the equations
$$X_{m} = \gamma_{m} X_{1}^{q^{m-1}} + \sum_{r=2}^{m} \beta_{m,r,i} \ \ (2 \leq m \leq d+1).$$
Since $X_{d+1} = X_{1}$, we can solve these equations in $\ks$.
This proved the existence part of this proposition.

Next we calculate the valuations of these solutions by induction on $i$.
We set $f_{i}(X_{1}) := \gamma_{d+1} X_{1}^{q^d} - X_{1} + \sum_{r=2}^{d+1} \beta_{d+1,r,i}$.
Since $a_{l,i} = 0$ for all $i < 0$, the inequality for the valuations in the statement of this proposition is true for $i < 0$.
Fix $i \geq 0$ and assume that the inequality in the statement of this proposition is true for integers lower than $i$.
It is clear that $v_{1}(\gamma_{d+1}) = v_{1}(c_{1}) = n$.
For $2 \leq r \leq d$, we have
\begin{align*}
v_{1}(\beta_{d+1,r,i}) &= n + q^{d+1-r} v_{1}(b_{r,i}) \\
&\geq n + q^{d+1-r} \min_{1 \leq j \leq n,i} \{ v_{1}(\binom{n}{j}) + q v_{1}(a_{r-1,i-j}), v_{1}(\binom{n}{i}) + q v_{1}(\alpha) \} \\
&\geq n + q^{d+1-r} \min_{1 \leq j \leq n,i} \{ -q^{r-1} \left( \frac{i-j}{q^{d}} + \frac{n}{q^{d}-1} \right), 0 \} \\
&\geq n + q^{d+1-r} \left( -q^{r-1} \left( \frac{i-1}{q^{d}} + \frac{n}{q^{d}-1} \right) \right) \\
&= n - i + 1 - \frac{q^{d} n}{q^{d}-1}.
\end{align*}
For $r = d + 1$, we have
\begin{align*}
v_{1}(\beta_{d+1,d+1,i}) &= v_{1}(b_{1,i}) \\
&\geq \min_{1 \leq j \leq n,i} \{ n - j + q v_{1}(a_{d,i-j}), n - i + q v_{1}(\alpha) + v_{1}(\binom{n}{i}) \} \\
&\geq \min_{1 \leq j \leq n,i} \{ n - j - q^{d} \left( \frac{i-j}{q^{d}} + \frac{n}{q^{d}-1} \right), 0 \} \\
&\geq n - i - \frac{q^{d} n}{q^{d}-1}.
\end{align*}
Thus we conclude that
$v_{1}(\sum_{r=2}^{d+1} \beta_{d+1,r,i}) \geq n - i - q^{d} n / (q^{d}-1)$.
By considering the Newton polygon of $f_{i}$,
we have $v_{1}(a_{1,i}) \geq - i / q^{d} - n / (q^{d}-1)$ for any root $a_{1,i}$ of $f_{i}$.
For $2 \leq r \leq m \leq d$, we have
$$v_{1}(\beta_{m,r,i}) = q^{m-r} v_{1}(b_{r,i}) \geq q^{m-r} \left( -q^{r-1} \left( \frac{i-1}{q^{d}} + \frac{n}{q^{d}-1} \right) \right)
= - q^{m-1} \left( \frac{i-1}{q^{d}} + \frac{n}{q^{d}-1} \right)$$
and
$$v_{1}(\gamma_{m} a_{1,i}^{q^{m-1}}) = q^{m-1} v_{1}(a_{1,i}) \geq - q^{m-1} \left( \frac{i}{q^{d}} + \frac{n}{q^{d}-1} \right).$$
Thus we have
$$v_{1}(a_{m,i}) = v_{1}(\gamma_{m} a_{1,i}^{q^{m-1}} + \sum_{r=2}^{m} \beta_{m,r,i})
\geq - q^{m-1} \left( \frac{i}{q^{d}} + \frac{n}{q^{d}-1} \right).$$
\end{proof}

The next proposition is proved by similar arguments as Proposition \ref{abp_l}.
\begin{prop}\label{abp_omega}
There exists an element
$\Omega_{v} = \Omega_{v}(t) = (\sum_{i} a_{l,i} (t-\lambda_{l}))_{l} \in \ks[t]_{v}^{\times} = \prod_{l}\ks[\![t-\lambda_{l}]\!]^{\times}$
which satisfies the equation
\begin{align}
\sigma(\Omega_{v}) = (t-\theta) \Omega_{v}. \label{eq_omega}
\end{align}
For any $l \in \mathbb{Z}/d$, $0 \leq m \leq d-1$ and $i \geq 0$, we have
$$v_{l}(a_{l+m,i}) = \frac{q^{m}}{q^{i d}(q^{d} - 1)}.$$
\end{prop}

\noindent
By Propositions \ref{abp_l} and \ref{abp_omega}, the infinite sums
$L_{\alpha,n}(\theta)$ and $\Omega_{v}(\theta)$ converge.

\begin{exam}\label{abp_carlitz}
We define the \textit{Carlitz motive} to be the $\vp$-module $C$ whose underlying $\ktv$-vector space is $\ktv$
and on which $\vp$ acts by
$$\vp(f) = (t - \theta) \sigma(f)$$
for each $f \in C$.
The equation (\ref{eq_omega}) means that the element $\Omega_{v}$ in Proposition \ref{abp_omega} is a period of $C$.
If we write $\Omega_{v} = (\Omega_{v,l})_{l} = (\sum_{i} a_{l,i} (t-\lambda_{l}))_{l}$,
then $[K_{d}(a_{l,0},a_{l,1},\dots) : K_{d}] = \infty$ by Proposition \ref{abp_omega}.
Thus $\Omega_{v,l}$ is transcendental over $\ktv = K_{d}(\!(t-\lambda_{l})\!)$.
Therefore we have that $\tdeg_{\ktv} \Lal = 1$ and $\Gamma_{C} = \mathbb{G}_{m}$.
\end{exam}

\section{Algebraic independence of formal polylogarithms}\label{sec_log}
In this section, we prove the algebraic independence of certain \lq\lq formal'' polylogarithms.
The proof of this theorem follows \cite{ChYu} and \cite{Papa} closely.
Let $(F, E, L)$ be a $\sigma$-admissible triple and $t, \theta \in E$ distinct elements.
Let $n, r$ be positive integers and $\alpha_{1}, \dots, \alpha_{r} \in E$ fixed elements.
Assume that $(F^{\times})_{\mathrm{tor}} \neq F^{\times}$,
and there exist elements $\Omega = (\Omega_{l})_{l} \in L^{\times}$ and
$L_{\alpha_{j}, n} = (L_{\alpha_{j}, n, l})_{l} \in L$ for each $j = 1, \dots, r$ such that
$\sigma(\Omega) = (t - \theta) \Omega$, $\Omega_{l}$ is transcendental over $E$ and
$\sigma(L_{\alpha_{j}, n}) = \sigma(\alpha_{j}) + L_{\alpha_{j}, n} / (t - \theta)^{n}$.
In the $v$-adic settings, such elements actually exist
if $\alpha_{1}, \dots, \alpha_{r} \in K^{\times}$ (cf.\ Section \ref{sec_abp}).
We set
$$\Phi := 
\begin{bmatrix}
(t-\theta)^{n} & & & \\
\sigma(\alpha_{1}) (t - \theta)^{n} & 1 & & \\
\vdots & & \ddots & \\
\sigma(\alpha_{r}) (t - \theta)^{n} & & & 1 \\
\end{bmatrix}
\textrm{ and }
\Psi :=
\begin{bmatrix}
\Omega^{n} & & & \\
\Omega^{n} L_{\alpha_{1}, n} & 1 & & \\
\vdots & & \ddots & \\
\Omega^{n} L_{\alpha_{r}, n} & & & 1 \\
\end{bmatrix}.
$$
Then we have $\sigma \Psi = \Phi \Psi$.
Therefore, if $M$ is the $\vp$-module over $E$ corresponding to $\Phi$,
then $M$ is $L$-trivial.
This type of $t$-motive is considered in \cite{ChYu} and \cite{Papa}.
Note that in $\infty$-adic case, $\Omega$ and $L_{\alpha, n}$ are constructed explicitly,
and $L_{\alpha, n}(\theta)$ is the $n$-th Carlitz polylogarithm of $\alpha$.
We define $\Gamma, \Gamma_{M}, Z, \Lal, \dots$ as in the previous sections for $M$, $\Phi$ and $\Psi$.
In particular, we have
$$\Lal = E(\Omega_{l}^{n}, L_{\alpha,n,l}, \dots, L_{\alpha,n,l}).$$
Furthermore, we assume that,
$\Gamma(F)$ is Zariski dense in $\Gamma$ or $\Lal / F$ is regular extension for each $l$.
Thus the natural immersion $\Gamma \rar \Gamma_{M}$ is an isomorphism by Theorem \ref{p-d_pi-isom}.

For each $F$-algebra $R$, we set
$$G(R) := \left\{\begin{bmatrix}
* & 0 & \cdots & 0 \\
* & 1 & & \\
\vdots & & \ddots & \\
* & & & 1 \\\end{bmatrix} \in \GL_{r+1}(R) \right\}.$$
Then $G$ is an algebraic group over $F$ and we have a natural inclusion $\Gamma \subset G$.
Let $X_{0}, \dots, X_{r}$ be the coordinates of $G$ such that
the first column of a general element of $G$ \lq\lq is''
$$\begin{bmatrix}
X_{0} & & & \\
X_{1} & 1 & & \\
\vdots & & \ddots & \\
X_{r} & & & 1 \\
\end{bmatrix}.$$
We have the exact sequence $1 \rar \mathbb{G}_{a}^{r} \rar G \rar \mathbb{G}_{m} \rar 1$,
here $\mathbb{G}_{a}^{r}$ is the subgroup scheme of $G$ with coordinates $(X_{1}, \dots, X_{r})$ and
$\mathbb{G}_{m}$ is the quotient of $G$ given by the projection $(X_{i}) \mapsto X_{0}$.
Let $C \in \phm_{E}^{L}$ be the one-dimensional $\vp$-module such that
$\vp(f) = (t - \theta) \sigma(f)$ for each $f \in C = E$.
Then we have the following exact sequence:
$$0 \rar C^{\otimes n} \rar M \rar \mathbf{1}^{r} \rar 0.$$
Thus $C^{\otimes n}$ is an object of $\TM$ and
we have the canonical surjection $\pi \colon \Gamma \cong \Gamma_{M} \rar \Gamma_{C^{\otimes n}} \cong \mathbb{G}_{m}$.
We set $V := \ker \pi$.
Then we have the commutative diagram
\[\xymatrix{
1 \ar[r] & V \ar[r] \ar@{^(->}[d] & \Gamma \ar[r]^{\pi} \ar@{^(->}[d] & \mathbb{G}_{m} \ar[r] \ar@{=}[d] & 1 \\
1 \ar[r] & \mathbb{G}_{a}^{r} \ar[r] & G \ar[r] & \mathbb{G}_{m} \ar[r] & 1, \\
}\]
where the rows are exact.

\begin{prop}\label{log_V}
The subgroup $V$ of $\mathbb{G}_{a}^{r}$ is defined by linear forms in $X_{1}, \dots, X_{r}$ with $F$ coefficients.
\end{prop}

\begin{proof}
Let $T \subset \Gamma_{\bar{F}}$ be a maximal torus and
$\bar{\pi} \colon \Gamma_{\bar{F}} \rar \mathbb{G}_{m, \bar{F}}$ be the base extension of $\pi$ to $\bar{F}$.
Then we have $\dim T = 1$ and $\bar{\pi}|_{T} \colon T \rar \mathbb{G}_{m, \bar{F}}$ is an isomorphism.
Thus $d\bar{\pi}$ is non-trivial and so is $d \pi$.
Hence we have the following exact sequence:
\[\xymatrix{
0 \ar[r] & \Lie V \ar[r] & \Lie \Gamma \ar[r] & \Lie \mathbb{G}_{m} \ar[r] & 0. \\
}\]
Since $\Gamma$ and $\mathbb{G}_{m}$ are smooth over $F$, we have the equalities
$\dim_{F} \Lie \Gamma = \dim \Gamma$ and $\dim_{F} \Lie \mathbb{G}_{m} = 1$.
Thus we have the equality $\dim_{F} \Lie V = \dim V$.
Therefore $V$ is smooth over $F$.
Thus it is enough to show that the space $V(\bar{F})$ is a linear space defined over $F$.
Let
$$\mu = \begin{bmatrix} 1 & 0 \\ v & I_{r} \\ \end{bmatrix} \in V(\bar{F}) \textrm{ and }
\alpha \in \bar{F}^{\times}$$
be any elements.
Since $\Gamma(\bar{F}) \rar \mathbb{G}_{m}(\bar{F})$ is surjective,
there exists an element $\gamma \in \Gamma(\bar{F})$ such that $\pi(\gamma) = \alpha$.
Then we have
$$V(\bar{F}) \ni \gamma^{-1} \mu \gamma = \begin{bmatrix} 1 & 0 \\ \alpha v & I_{r} \\ \end{bmatrix}.$$
Thus $V(\bar{F})$ is a linear subspace of $\mathbb{G}_{a}^{r}(\bar{F})$.
Since $V$ is defined over $F$, $V$ is defined by linear forms in $X_{1}, \dots, X_{r}$ with $F$ coefficients.
\end{proof}

Since $V$ is smooth and $\mathrm{H}^{1}(F, V) = 1$, we have the exact sequence
\[\xymatrix{
1 \ar[r] & V(F) \ar[r] & \Gamma(F) \ar[r] & \mathbb{G}_{m}(F) \ar[r] & 1. \\
}\]
By the assumption on $F$, there exists an element $b_{0} \in F^{\times} \smallsetminus (F^{\times})_{\mathrm{tor}}$.
By the above sequence, there exists an element
$$\gamma =
\begin{bmatrix} b_{0} & & & \\ b_{1} & 1 & & \\ \vdots & & \ddots & \\ b_{r} & & & 1 \\ \end{bmatrix}
\in \Gamma(F).$$
We fix such $b_{0}$ and $\gamma$.
For each $F$-algebra $R$ and $a \in R^{\times}$, we set
$$\gamma_{a} := \begin{bmatrix} a & & & \\ \frac{b_{1}}{b_{0} - 1}(a - 1) & 1 & & \\
\vdots & & \ddots & \\ \frac{b_{r}}{b_{0} - 1}(a - 1) & & & 1 \\ \end{bmatrix}.$$
Then for each $a, b \in R^{\times}$ and $m \in \mathbb{Z}$,
we have $\gamma_{a} \gamma_{b} = \gamma_{a b}$ and $\gamma^{m} = \gamma_{b_{0}^{m}}$.
Hence we have
$\overline{\langle \gamma \rangle} = ( R \mapsto \{ \gamma_{a} | a \in R^{\times} \} )$, a line in $\Gamma$.
We set $\Gamma' := \overline{\langle V, \gamma \rangle} \subset \Gamma$ and $s := r - \dim V$.
We claim that $\Gamma' = \Gamma$.
Indeed, let
\begin{align}
F_{i} = \sum_{j=1}^{r} c_{i,j} X_{j} \in F[X_{1}, \dots, X_{r}] \ \ (i = 1, \dots, s) \label{eq_V}
\end{align}
be linear forms defining $V$.
For each $i$, we set
$$G_{i} := (b_{0} - 1) F_{i}(X_{1}, \dots, X_{r}) - F_{i}(b_{1}, \dots, b_{r})(X_{0} - 1) \in F[X_{0}, \dots, X_{r}].$$
Then we can verify that $G_{1}, \dots, G_{s}$ define $\Gamma'$ in $\GL_{r+1}$
and $\Gamma'$ is an algebraic group.
Since $V \subset \Gamma' \subset \Gamma$ and $\Gamma' \rar \mathbb{G}_{m}$ is surjective,
we have $\Gamma' = \Gamma$.
Thus we have the following proposition:

\begin{prop}\label{log_Gamma}
The algebraic group $\Gamma$ is defined by the linear polynomials $G_{1}, \dots, G_{s}$ in $\GL_{r+1/F}$.
\end{prop}

Since $Z_{\bar{E}} \cong \Gamma_{\bar{E}}$ and $Z$ is defined over $E$,
$Z$ is defined by linear polynomials over $E$, and there exists an $E$-valued point
$$\xi = \begin{bmatrix} f_{0} & & & \\ f_{1} & 1 & & \\ \vdots & & \ddots & \\ f_{r} & & & 1 \\ \end{bmatrix} \in Z(E).$$
We fix such $\xi$.
Then we have $Z = \xi \cdot \Gamma_{E}$.
Set $f_{i}' := G_{i}(f_{0}, \dots, f_{r}) f_{0}^{-1} \in E$ and
$H_{i} := G_{i}(X_{0}, \dots, X_{r}) - X_{0}f_{i}' \in E[X_{0}, \dots, X_{r}]$.
Then $H_{1}, \dots, H_{s}$ are defining polynomials for $Z$.
If we set $g_{i} := \sum_{j=1}^{r} c_{i,j} b_{j}$, then we have
$$H_{i} = (b_{0} - 1) \sum_{j=1}^{r} c_{i,j} X_{j} + g_{i} - (g_{i} + f_{i}') X_{0}.$$
Since $\Psi_{l} \in Z(\Sl)$ for each $l$, we have
$$(b_{0} - 1) \sum_{j=1}^{r} c_{i,j} L_{\alpha_{j},n,l} + g_{i} \Omega_{l}^{-n} - (g_{i} + f_{i}') = 0$$
for each $l$ and $i$.
Set $B := (c_{i,j})_{i,j} \in \Mat_{s \times r}(F)$.
By the definition of $c_{i,j}$ (\ref{eq_V}), the rank of $B$ is $s = r - \dim V$.
Set
$$P = \begin{bmatrix} P_{1} \\ \vdots \\ P_{s} \\ \end{bmatrix} :=
\begin{bmatrix} (b_{0} - 1) c_{1,1} & \dots & (b_{0} - 1) c_{1,r} & g_{1} & -(g_{1} + f_{1}') \\
\vdots & & \vdots & \vdots & \vdots \\
(b_{0} - 1) c_{s,1} & \dots & (b_{0} - 1) c_{s,r} & g_{s} & -(g_{s} + f_{s}') \\ \end{bmatrix} \in \Mat_{s \times (r + 2)}(E),$$
the coefficients matrix of the above equations.
Then the rank of $P$ is also $s$.
We interested in
$$N_{l} := \langle L_{\alpha_{1},n,l}, \dots, L_{\alpha_{r},n,l}, \Omega_{l}^{-n}, 1 \rangle_{E} \subset \Lal.$$
This is the image of the $E$-linear map
$$\beta_{l} \colon E^{r+2} \rar \Lal; \ (x_{1}, \dots, x_{r+2}) \mapsto \sum_{j=1}^{r} x_{j} L_{\alpha_{j},n,l} + x_{r+1} \Omega_{l}^{-n} + x_{r+2}.$$
Since $P_{i} \in \ker \beta_{l}$ for each $i$, we have the inequality $\dim_{E} \ker \beta_{l} \geq s$.
Thus we have $\dim_{E} N_{l} \leq r + 2 - s = \dim V + 2 = \dim \Gamma + 1 = \tdeg_{E} \Lambda_{l'} + 1$ for each $l$ and $l'$.
On the other hand, it is clear that $\dim_{E} N_{l} \geq \tdeg_{E} \Lal = \tdeg_{E} \Lambda_{l'}$.
Thus we have the following theorem:

\begin{theo}\label{log_deg-dim}
For each $l$ and $l'$, we have $\tdeg_{E} \Lambda_{l'} \leq \dim_{E} N_{l} \leq \tdeg_{E} \Lambda_{l'} + 1$.
\end{theo}

\begin{coro}\label{log_indep}
If $L_{\alpha_{1},n,l}, \dots, L_{\alpha_{r},n,l}, 1$ are linearly independent over $E$ for some $l$,
then $L_{\alpha_{1},n,l'}, \dots, L_{\alpha_{r},n,l'}$ are algebraically independent over $E$ for each $l'$.
\end{coro}

\begin{proof}
Note that since $\Lambda_{l'} = E(\Omega_{l'}^{n}, L_{\alpha_{1},n,l'}, \dots, L_{\alpha_{r},n,l'})$,
we have $\tdeg_{E} \Lambda_{l'} \leq r + 1$.
By the assumption, we have $r + 1 \leq \dim_{E} N_{l} \leq r + 2$.

Assume that $\dim_{E} N_{l} = r + 2$.
Then $\tdeg_{E} \Lambda_{l'} < \dim_{E} N_{l}$.
By Theorem \ref{log_deg-dim}, we have $\dim_{E} N_{l} = \tdeg_{E} \Lambda_{l'} + 1$.
Thus we have $\tdeg_{E} \Lambda_{l'} = r + 1$ and
$\Omega_{l'}^{n}, L_{\alpha_{1},n,l'}, \dots, L_{\alpha_{r},n,l'}$ are algebraically independent over $E$.

On the other hand, assume that $\dim_{E} N_{l} = r + 1$.
By the assumption, we can write $\Omega_{l}^{-n}$ as a linear combination of $L_{\alpha_{1},n,l}, \dots, L_{\alpha_{r},n,l}, 1$ over $E$.
In particular, we have $\Omega_{l}^{n} \in E(L_{\alpha_{1},n,l}, \dots, L_{\alpha_{r},n,l})$.
Letting $\sigma$ act on this relation, we have
$$(t - \theta)^{n} \Omega_{l+1}^{n} \in
\sigma(E)\left(\sigma(\alpha_{1}) + \frac{L_{\alpha_{1},n,l+1}}{(t - \theta)^{n}}, \dots, \sigma(\alpha_{r}) + \frac{L_{\alpha_{r},n,l+1}}{(t - \theta)^{n}})\right).$$
Thus for each $l'$, we have $\Omega_{l'}^{n} \in E(L_{\alpha_{1},n,l'}, \dots, L_{\alpha_{r},n,l'})$
and $\Lambda_{l'} = E(L_{\alpha_{1},n,l'}, \dots, L_{\alpha_{r},n,l'})$.
By Theorem \ref{log_deg-dim}, we have $\tdeg_{E} \Lambda_{l'} \geq \dim_{E} N_{l} - 1 = r$.
Thus $\tdeg_{E} \Lambda_{l'} = r$ and
$L_{\alpha_{1},n,l'}, \dots, L_{\alpha_{r},n,l'}$ are algebraically independent over $E$.
\end{proof}

\end{document}